\documentclass{amsart}

\title{Limits of multivariate elliptic hypergeometric biorthogonal functions}
\author{Fokko J. van de Bult and Eric M. Rains}

\usepackage{fullpage}
\usepackage{tikz} 
\usepackage{multirow}

\newtheorem{theorem}{Theorem}[section]
\newtheorem{definition}[theorem]{Definition}
\newtheorem{proposition}[theorem]{Proposition} 
\newtheorem{corollary}[theorem]{Corollary}
\newtheorem{lemma}[theorem]{Lemma}

\newcommand{\binon}[2]{\genfrac{\langle}{\rangle}{0pt}{}{#1}{#2}}

\begin{document}

\begin{abstract}
In this article we extend the results of \cite{vdBRuniv} to the multivariate setting. In \cite{vdBRuniv}, we determined which families of biorthogonal functions arise as limits  from the elliptic hypergeometric biorthogonal  functions from \cite{S} when $p\to 0$. Here we show that the classification of the possible limits of the $BC_n$ type multivariate biorthogonal functions from \cite{Rainstrafo} and \cite{RainsBCn} is identical to the univariate classification. That is, for each univariate limit family there exists a multivariate extension, and in particular we obtain multivariate versions for all elements of the $q$-Askey scheme. For the Askey-Wilson polynomials these are the Koornwinder polynomials, and the multivariate versions of the Pastro polynomials form a two-parameter family which include the Macdonald polynomials.
\end{abstract}

\maketitle

In \cite{vdBRuniv} the authors considered the limits as $p\to 0$ of the univariate elliptic hypergeometric biorthogonal functions 
studied by Spiridonov and Zhedanov in \cite{S}, \cite{SZ1} and \cite{SZ2}. It turned out there are 38 distinct families of  biorthogonal basic hypergeometric rational functions which appear as limits. Moreover, the degeneration scheme included as a subset the $q$-Askey scheme of orthogonal basic hypergeometric polynomials. 

In \cite{RainsBCn} and \cite{Rainstrafo} the second author considered a $BC_n$-symmetric multivariate extension of these biorthogonal functions and established their basic properties. These properties include generalizations of all Macdonald conjectures: Explicit formulas are given for the squared norms, evaluation of the biorthogonal functions at suitable geometric sequences, and evaluation symmetry between the spectral and geometric parameters. 

In this paper we want to determine what the possible (basic hypergeometric) limits of these multivariate biorthogonal functions are as $p\to 0$. These limits of course depend on how the parameters (other than $p$ itself) behave as $p\to 0$, so to make the question more explicit we impose conditions on how they depend on $p$, similar as in \cite{vdBRuniv}. Indeed, if we write
$\mathcal{R}_{\lambda}^{(n)}(z_i;t_r;u_r;q,t;p)$ for the biorthogonal functions (as in \cite{Rainstrafo}), where we have $n$ variables $z_i$, four parameters $t_r$ and two parameters $u_r$, we set $z_i \to z_ip^{\zeta}$, $t_r \to t_r p^{\alpha_r}$ and $u_r \to u_r p^{\gamma_r}$, while forcing $q$ and $t$ to be independent of $p$. It should be noted that all variables $z_i$ exhibit identical behavior as $p\to 0$ (i.e., there is just one $\zeta$). The limit we arrive at will now obviously depend on the values of $\zeta$, $\alpha_r$ and $\gamma_r$. 

As in \cite{vdBRuniv} there are essentially two conditions on the limits which determine whether the limits themselves also form a family of biorthogonal functions. First of all we must ensure that the limits are still $z$-dependent (and form a linearly independent set of functions), and secondly we must be able to take the limit in the squared norm formula. The major part of this paper consists in showing that the condition that the limits are $z$-dependent is equivalent to the condition that the related univariate biorthogonal functions are $z$-dependent. It is straightforward to see that the condition that the squared norm formula has a proper limit is identical to that condition in the univariate case. As a corollary we obtain that the degeneration scheme of the multivariate biorthogonal functions is identical to the degeneration scheme in the univariate case derived in \cite{vdBRuniv}. In particular, for each of the limits obtained in \cite{vdBRuniv} there exists a multivariate analogue. We  give explicit measures for which these families are biorthogonal in \cite{vdBRmeas}.

One might expect that there also exist interesting limits if one allows $t$ to vary with $p$. We refrain from considering those cases here, as our choice of fixing $t$ ensures that the combinatorics of the multivariate limits equals that of the univariate limits as discussed in the previous paragraph. 

It comes as no surprise that the Koornwinder polynomials appear as the multivariate analogue of the Askey-Wilson polynomials, and we obtain a multivariate analogue of all other families of polynomials in the $q$-Askey scheme. All of those families of  multivariate orthogonal polynomials can be obtained as limits of the Koornwinder polynomials. Another interesting special case arises as the multivariate analogues of the Pastro polynomials. These are biorthogonal polynomials, but specializing them in a proper way reduces them to the original Macdonald polynomials. 

In \cite{RainsBCn} and \cite{Rainstrafo} two different approaches are used to define the multivariate biorthogonal functions. In \cite{Rainstrafo} they are defined using difference and integral raising operators, in \cite{RainsBCn} they are expanded in terms of interpolation functions, which themselves are defined by vanishing conditions. Neither of these two methods is well-suited for taking limits directly. Thus the approach we take in this article is to define the interpolation functions using the branching rule. This gives us an explicit expression of the biorthogonal functions as a finite sum of products of theta functions. We want to take the limit in this expression by taking the limit of each summand, however this will sometimes lead to unwanted cancellation (where the sum of the limits of the summands vanishes). We will show that the symmetries of the biorthogonal function allow us to find at least one expression for the biorthogonal function in which this does not happen, thus giving us an explicit method of finding the limits. 

The article is organized as follows. We start with a section on notation, followed by a section giving the definition of the interpolation functions and the associated generalized binomial coefficients. In Section \ref{secbiortho} we give the definition of the biorthogonal functions and its most important properties. Section \ref{secliminterpolation} describes how we can obtain limits of the interpolation functions. The next section proves the main result about the limits of biorthogonal functions: They exist as linearly independent functions of $z$ only if their univariate counterparts do. In Section \ref{secpastro} we consider the specific limit to multivariate Pastro polynomials and their special case, the Macdonald polynomials. The appendix gives explicit expressions for the limits of the interpolation functions.

\section{Notation}\label{secnot}

\subsection{Univariate $q$-symbols}

We say a function $f(x;z)$ is written multiplicatively in $x$ if the presence of multiple parameters at the place of $x$ indicates a product; and if $\pm$ symbols in those parameters also indicate a product over all possible combinations of $+$ and $-$ signs. For example
\begin{align*}
f(x_1, x_2, \ldots, x_n;z) &= \prod_{r=1}^n f(x_i;z), \\ f(x^{\pm 1} y^{\pm 1};z) &= 
f(xy;z)f(x/y;z)f(y/x;z)f(1/xy;z).
\end{align*}

Now we define the $q$-symbols and their elliptic analogues as in \cite{GR}. Let $0<|q|,|p|<1$ and set
\begin{align*}
(x;q) &= \prod_{r=0}^\infty (1-xq^r), &
(x;q)_m &= \prod_{r=0}^{m-1} (1-xq^r), & (x;p,q) &= \prod_{r,s\geq 0} (1-xp^rq^s) \\
\theta(x;p) &= (x,p/x;p),&
\theta(x;q;p)_m &= \prod_{r=0}^{m-1} \theta(xq^r;p), & 
\Gamma(x;p,q) & = \prod_{i,j\geq 0} \frac{1-p^{i+1}q^{j+1}/x}{1-p^i q^j x}.
\end{align*}
All these functions are written multiplicatively in $x$. Note that the terminating product $(x;q)_m$ is also defined if $|q|\geq 1$. Likewise $\theta(x;q;p)_m$ is defined for all $q$, though we must still insist on $|p|<1$.

%

%
%

\subsection{Partitions}
We use the notations of \cite{RainsBCn} for partitions, which is the notation from Macdonald's book \cite{Macdonald} with some additions. 
If $\lambda \subset m^n$ then we write $m^n-\lambda$ for the complementary partition, given by 
\[
(m^n-\lambda)_i = \begin{cases} m-\lambda_{n+1-i}  & 1\leq i\leq n \\ 0 & i>n \end{cases}
\]
Moreover, if $\ell(\lambda)\leq n$ we define $m^n+\lambda$ to be the partition
\[
(m^n+\lambda)_i = \begin{cases} m+\lambda_i & 1\leq i\leq n \\0 & i>n \end{cases}
\]
Similarly, if $\lambda_1\leq m$ we define $m^n\cdot \lambda$ to be the partition
\[
(m^n \cdot \lambda)_i = \begin{cases} m & 1\leq i\leq n, \\ \lambda_{i-n} & i>n. \end{cases}
\]
Stated differently: $m^n \cdot \lambda = (m^n+\lambda')'$. 

We define the relation $\prec_m$ by setting $\kappa \prec_m \lambda$ if $\kappa \subset \lambda \subset m^n+\kappa$ for sufficiently large $n$ (for example $n= \max(\ell(\kappa),\ell(\lambda))$). Likewise we define $\prec_m'$ by setting $\kappa \prec_m' \lambda$ if and only if $\kappa' \prec_m \lambda'$. In both cases we omit the subscript if $m=1$. 
Recall that a chain $0 =\lambda^{(0)} \prec' \lambda^{(1)} \prec' \lambda^{(2)} \prec' \cdots \prec' \lambda^{(n)}$ corresponds to a semistandard Young tableau of shape $\lambda^{(n)}$, where the entries of the diagram are determined by writing $k$ in the boxes in the strip $\lambda^{(k)}/ \lambda^{(k-1)}$. 

Some convenient numbers associated with $\lambda$ are
\begin{align*}
|\lambda|&=\sum_i \lambda_i \\
n(\lambda) &= \sum_{i} \binom{\lambda_i'}{2} = \sum_{(i,j)\in \lambda} (i-1) = \frac12 \sum_{(i,j) \in \lambda} (\lambda_j'-1)
\end{align*}
Here we use $\sum_{(i,j)\in \lambda}$, which means we sum over all boxes in the Young diagram, i.e. we sum over 
$1\leq i\leq l(\lambda)$ and all $1\leq j\leq \lambda_i$. A similar notation is used for products. 

In the entire article we will use $n$ for the number of variables $z_i$, which means that our partitions usually satisfy $\ell(\lambda)\leq n$. (From context it should always be clear when we use $n$ as number of variables and when we use it for the function $n(\lambda)$.)

\subsection{Multivariate $q$-symbols}

A meromorphic function $f(z_i,\ldots, z_n)$ is called a $BC_n$-symmetric $p$-abelian function if it satisfies
\begin{itemize}
\item $f$ is invariant under permutations of the $z_i$;
\item $f$ is invariant under replacing any one of the $z_i$ by $1/z_i$;
\item $f$ is invariant under replacing any one of the $z_i$ by $pz_i$.
\end{itemize}
We define the space $A^{(n)}(u_0;p,q)$ as the space of all $BC_n$-symmetric $p$-abelian functions $f$ such that
\[
\prod_{i=1}^n \theta(pq z_i^{\pm 1}/u_0;q;p)_m f(\ldots,z_i,\ldots)
= \prod_{i=1}^n \frac{\Gamma(u_0 z_i^{\pm 1})}{\Gamma(u_0 q^{-m} z_i^{\pm 1})} f(\ldots,z_i,\ldots)
\]
is holomorphic for $z\in (\mathbb{C}^*)^n$ for sufficiently large $m$. That is, $f$ can only have poles at the points $u_0 q^{-l} p^k$ and
$u_0^{-1} q^l p^k$ for $k\in \mathbb{Z}$ and $1\leq l\leq m$, and these poles must be simple. It should be noted that 
$A^{(1)}(u_0;p,q) = A(u_0;p,q)$ as defined in \cite{vdBRuniv}.

Let us now define the $C$-symbols (also written multiplicatively in $x$).
\begin{align}
\label{eqdefc0}C_{\lambda}^0(x;q,t;p) &= \prod_{(i,j)\in\lambda} \theta(q^{j-1}t^{1-i}x;p) &\tilde C_{\lambda}^0(x;q,t) &= \prod_{(i,j)\in\lambda} (1-q^{j-1}t^{1-i}x) \\
\label{eqdefcm}C_{\lambda}^-(x;q,t;p) &= \prod_{(i,j)\in\lambda} \theta(q^{\lambda_i-j}t^{\lambda_j'-i}x;p) & \tilde C_{\lambda}^-(x;q,t) &= \prod_{(i,j)\in\lambda} (1-q^{\lambda_i-j}t^{\lambda_j'-i}x)\\
\label{eqdefcp} C_{\lambda}^+(x;q,t;p) &= \prod_{(i,j)\in\lambda} \theta(q^{\lambda_i+j-1}t^{2-\lambda_j'-i}x;p) &
\tilde C_{\lambda}^+(x;q,t) &= \prod_{(i,j)\in\lambda} (1-q^{\lambda_i+j-1}t^{2-\lambda_j'-i}x)
\end{align}
The elliptic $C_{\lambda}$ are as in \cite{RainsBCn}, while the $\tilde C_{\lambda}$ are the 
$C_{\lambda}$ from \cite{RainsBCnpoly}.
%
%
%
%
%
Finally we define the $\Delta$-symbols by
\[
\Delta_\lambda^0(a~|~b;q,t;p) = \frac{C_{\lambda}^0(b;q,t;p)}{C_\lambda^0(pqa/b;q,t;p)},
\] 
which is written multiplicatively in $b$ and
\[
\Delta_\lambda(a~|~b_1,\ldots, b_r;q,t;p) = \Delta_\lambda^0(a~|~b_1, \ldots, b_r;q,t;p) \frac{C_{2\lambda^2}^0(pqa;q,t;p)}{C_\lambda^-(pq,t;q,t;p) C_{\lambda}^+(a,pqa/t;q,t;p)}
\]
which is emphatically not written multiplicatively. Here $2\lambda^2$ denotes the partition with $(2\lambda^2)_i = 2(\lambda_{\lceil i/2\rceil})$. A $q$-hypergeometric version of $\Delta_{\lambda}$ is defined by 
\[
\tilde \Delta_{\lambda}^{(n)}(a;q,t) = \frac{\tilde C_{2\lambda^2}^0(aq;q,t) \tilde C_{\lambda}^0(t^n;q,t)}{
\tilde C_{\lambda}^0(aq/t^n) \tilde C_{\lambda}^-(q,t;q,t) \tilde C_{\lambda}^+(a,aq/t;q,t)}
\left(-\frac{1}{a^2q^2t^{n-1}}\right)^{|\lambda|} q^{-3n(\lambda')} t^{5 n(\lambda)} 
\]
Whenever no confusion is possible we omit the $;q,t;p$ or the $;q,t$ from the arguments.

The $C_\lambda$'s are multivariate analogues of the theta Pochhammer symbols, while the $\tilde C_{\lambda}$'s are multivariate versions of $q$-Pochhammer symbols. The $\Delta_{\lambda}$ and $\tilde \Delta_{\lambda}$ correspond univariately to the summands of a very well poised series, indeed 
%
\begin{align*}
\Delta_l(a~|~b_1,\ldots,b_r;q,t;p) &=
\frac{\theta(pqa;q;p)_{2l} \theta(\frac{pqa}{t};q;p)_{l}}{\theta(pq,t,aq^l;q;p)_l }
\prod_{s=1}^r \frac{\theta(b_s;q;p)_l}{\theta(\frac{pqa}{b_s};q;p)_l}, \\ 
\tilde \Delta_{l}^{(1)}(a;q,t) &= \frac{1-aq^{2l}}{1-a} \frac{(a;q)_{l}}{(q;q)_l} \left( -\frac{1}{a^2q^2}\right)^l q^{-3\binom{l}{2}}.
\end{align*}
\subsection{Transformations of generalized $q$-symbols}
It is convenient to write down a few elementary transformation formulas for these functions, analogues of some identities for theta Pochhammer symbols. 
The following expressions can all be obtained from the two elementary symmetries $\theta(px;p)=\theta(1/x;p)=-\frac{1}{x}\theta(x;p)$.
\begin{align}
C_{\lambda}^0(px;q,t;p) &= C_{\lambda}^0(1/x;1/q,1/t;p) = C_{\lambda;q,t;p}^0(x;q,t;p) \left(-\frac{1}{x}\right)^{|\lambda|} q^{-n(\lambda')} t^{n(\lambda)}, \label{eqc0p} \\
C_{\lambda}^-(px;q,t;p) &= C_{\lambda}^-(1/x;1/q,1/t;p) = C_{\lambda;q,t;p}^-(x) \left(-\frac{1}{x}\right)^{|\lambda|} q^{-n(\lambda')} t^{-n(\lambda)}, \label{eqcmp}\\
C_{\lambda}^+(px;q,t;p) & = C_{\lambda}^+(1/x;1/q,1/t;p) =  C_{\lambda;q,t;p}^+(x) \left(-\frac{1}{qx}\right)^{|\lambda|} q^{-3n(\lambda')} t^{3n(\lambda)}. \label{eqcpp}
\end{align}

Likewise we can find shifting formulas for the $\Delta$ functions:
\begin{align}
\Delta_{\lambda}^0(a~|~pb, \ldots, v_i,\ldots) & = \Delta_{\lambda}^0(a~|~b,\ldots,v_i,\ldots) \left( \frac{1}{aq}\right)^{|\lambda|} q^{-2n(\lambda')}t^{2n(\lambda)} \\
\Delta_{\lambda}^0(\frac{a}{p}~|~ b_1, \ldots, b_r) &= \Delta_{\lambda}^0(a~|~b_1,\ldots, b_r)\left(\frac{\prod_i b_i}{(-aq)^r}\right)^{|\lambda|}
q^{-r n(\lambda')} t^{r n(\lambda)} \\
\Delta_{\lambda}(a~|~pb, \ldots, v_i, \ldots) &= \Delta_{\lambda}(a~|~b,\ldots,v_i,\ldots) \left( \frac{1}{aq}\right)^{|\lambda|} 
q^{-2n(\lambda')}t^{2n(\lambda)} \\
\Delta_{\lambda}(\frac{a}{p}~|~ b_1, \ldots, b_r) &= \Delta_{\lambda}(a~|~b_1,\ldots, b_r)\left(\frac{pq}{t} \frac{\prod_i b_i}{(-aq)^{r-2}}\right)^{|\lambda|}
q^{(2-r) n(\lambda')} t^{(r-2) n(\lambda)} 
\end{align}
We'd like to remark that $\Delta_{\lambda}^0(a~|~b_1,\ldots,b_{r})$ is invariant if we multiply each individual $b_j$ by an integer multiple of $p$, while keeping the product $\prod_r b_r$ fixed. Moreover, if $r$ is even, then $\Delta_{\lambda}^0$ is invariant if we multiply $a$ and the $b_j$'s by integer multiples of $p$, as long as the balancing condition $\prod_i b_i = (apq)^{r/2}$ holds (both before and after the $p$-shift). Similarly, as long as the balancing condition $pq \prod_{i} b_i=t (apq)^k$ holds $\Delta_{\lambda}(a~|~b_1,\ldots,b_{2k+2})$ remains invariant under multiplication of the parameters by integer powers of $p$.

\subsection{Power series in $p$}
Most functions we are interested in are elements of the field $M(x)$, defined in \cite[Section 2]{vdBRuniv}. This is a field of (multivariate) meromorphic functions in the variables $x=(x_1,x_2,\ldots)$, which can be expressed as power series 
$f=\sum_{t\in T} a_t(x) p^t$ for some discrete set $T$, which is bounded from below, with coefficients $a_t$, which are rational functions in $x$. The valuation of such a series is $val(f) = \min_{t\in T}t$ and the leading coefficient is given by $lc(f) = a_{val(t)}$. Since we are interested in the behavior as $p\to 0$, we think of the valuation as describing the size of $f$ as $p\to 0$, while the leading coefficient gives the limit of $f$ (after proper rescaling). The conditions on the space imply that this limit is always uniform on compact sets outside the zero-set of some polynomial in $x$. Moreover, due to some extra conditions placed on the rational functions $a_t$ we obtained the following iterated limit property \cite[Proposition 2.3]{vdBRuniv}.
\begin{proposition}\label{propiteratedlim}
Let $f\in M(x)$, write $p^u x = (p^{u_1}x_1,p^{u_2}x_2,\ldots)$. Then for small enough $\epsilon>0$ and any $u$ we have
\[
lc(lc(f)(p^u x)) = lc(f(p^{\epsilon u}x)), \qquad val(f)+ \epsilon \  val(lc(f)(p^u x)) = val(f(p^{\epsilon u} x)).
\]
\end{proposition}
With as a corollary the following important result on the valuation of a sum of two terms
\begin{corollary}\label{corsumequalval}
Let $f,g\in M(x)$ and define $h=f+g$. 
\begin{itemize}
\item If $val(f)<val(g)$, then $val(h)=val(f)$ and $lc(h)=lc(f)$. 
\item If $val(f)=val(g)$, and there exists a $u$ such that 
for all small enough $\epsilon>0$ we have $val(f(p^{\epsilon u}x))< val(g(p^{\epsilon u}x))$. Then $val(h)=val(f)$ and $lc(h)=lc(f)+lc(g)$. 
\end{itemize}
\end{corollary}

\subsection{Limits of generalized $q$-symbols} 
Of course the $q$-symbols discussed before are elements of the field $M(x)$, and since every function appearing in this article is built using these $q$-symbols, they are elements of $M(x)$ as well. Let us now discuss the valuations and leading coefficients of the elliptic $q$-symbols.
%

For ordinary theta functions we have
\[
val(\theta(xp^{\alpha};p)) =\frac12 \{\alpha\}(\{\alpha\}-1)- \frac12 \alpha(\alpha-1), \qquad 
lc(\theta(xp^{\alpha};p)) = \begin{cases} (1-x) \left(-\frac{1}{x}\right)^{ \alpha } & \alpha \in \mathbb{Z} \\
\left(-\frac{1}{x}\right)^{\lfloor \alpha \rfloor} & \alpha \not \in \mathbb{Z}, \end{cases}
\]
where $\{\alpha\} = \alpha-\lfloor \alpha \rfloor$ denotes the fractional part of $\alpha$. Note that $val(\theta(xp^{\alpha};p))$ is a continuous piecewise linear function in $\alpha$. The valuations and leading coefficients of the $C$-symbols are direct consequences of this. While a general formula for the leading coefficient is easily given, it becomes rather complex as it changes for the different $C$-symbols. Thus we refer to the shifting formulas
\eqref{eqc0p} to note that it suffices to give the results for $0\leq \alpha<1$. We have
\begin{equation}\label{eqlimc}
val(C_{\lambda}^{\epsilon}(xp^{\alpha})=|\lambda| (\frac12\{\alpha\}(\{\alpha\}-1) - \frac12\alpha(\alpha-1)), \quad (\alpha\in \mathbb{R}), \qquad 
lc(C_{\lambda}^{\epsilon}(xp^{\alpha})=\begin{cases} \tilde C_{\lambda}^{\epsilon}(x) & \alpha=0 \\\
1 & 0<\alpha<1, \end{cases}
\end{equation}
where $\epsilon=0$, $+$, or $-$.

To take limits of $\Delta_{\lambda}^0$ it is often most convenient to express it in terms of $C_{\lambda}^0$, and take the limits of the $C_{\lambda}^0$'s. One of the important reasons we so often use the $\Delta_{\lambda}^0$ is that it is elliptic (under the balancing condition given above). After taking the limit, we cannot shift by $p$ anymore, so ellipticity becomes a non-existent concept, thus diminishing the usefulness of this notation.

As for $\Delta_{\lambda}$ we'll only consider $\Delta_{\lambda}(a p^{\alpha}~|~t^n ;q,t;p)$. It turns out that every instance of $\Delta_{\lambda}$ we encounter has $t^n$ as one of its $b$-parameters. Moreover the quotient of any $\Delta_{\lambda}$ and this one is a $\Delta_{\lambda}^0$ and we can express its limits in terms of $\tilde C_{\lambda}^0$'s as described above. Thus writing down the valuation and leading coefficient of this specific $\Delta_{\lambda}$ suffices to be able to obtain the limits of the general case. We assume $\ell(\lambda) \leq n$, as otherwise $\Delta_{\lambda}(a ~|~ t^n)=0$ identically.
%
\begin{align*}
val(\Delta_{\lambda}(ap^{\alpha}~|~ t^n) & = -2\alpha |\lambda|, \quad (0\leq \alpha<1), \\ 
lc(\Delta_{\lambda}(ap^{\alpha}~|~ t^n) )&= \begin{cases}
\tilde \Delta_{\lambda}^{(n)}(a;q,t) & \alpha=0, \\
\frac{\tilde C_{\lambda}^0(t^n)}{\tilde C_{\lambda}^-(q,t;q,t)}
\left(-\frac{1}{a^2q^2t^{n-1}}\right)^{|\lambda|} q^{-3 n(\lambda') } t^{5n(\lambda)}
  & 0<\alpha<1.
\end{cases}
\end{align*}

We would like to finish this section by making the following observation. Notice that the leading coefficients of 
these terms, only depend on whether $\alpha=0$ or $0<\alpha<1$. For general $\alpha$ it then follows that the leading coefficients  $lc(C_{\lambda}^{\epsilon}(p^{\alpha}x) )$ and $lc(\Delta(ap^{\alpha}~|~t^n))$ only depend on $\alpha$ through the  component of $\mathbb{R}$ which contains $\alpha$ if we cut $\mathbb{R}$ at the integers (i.e., write $\mathbb{R} = \mathbb{Z} \cup \bigcup_{n\in \mathbb{Z}} (n,n+1)$). Moreover, the leading coefficients associated to two $\alpha$'s in different components, which are related to each other by an integer shifts (i.e. either both $\alpha$'s are integers, or both are non-integers), differ by a  monomial factor (in $x$, $q$ and $t$).

\section{Interpolation functions and binomial coefficients}\label{secinterpolation}
In this section we give a recursive definition of the interpolation functions, binomial coefficients and biorthogonal functions from \cite{RainsBCn} and \cite{Rainstrafo}. We use a recursive definition here as this treatment is more suited for taking limits. We omit most of the proofs; all proofs are given in the referenced articles. Some of the basic results follow from a simple recursion (together with often somewhat tedious calculations). Many deeper results, however, are quite difficult to prove in the presentation we give here, so we are happy to simply observe that our functions equal the functions defined in the cited papers, and thus must satisfy the same properties.

We begin by defining
\begin{definition}\label{defbinon}
The binomial coefficient $\binon{\lambda}{\mu}_{[a,t];q,t;p}$ vanishes unless $\mu \prec'\lambda$, in which case
\begin{multline*}
\binon{\lambda}{\mu}_{[a,t];q,t;p} = \prod_{\substack{(i,j)\in \lambda \\ \lambda_j'=\mu_j'}}
\frac{\theta(q^{\lambda_i+j-1}t^{2-\lambda_j'-i} a;p)}{\theta(q^{\mu_i-j}t^{\mu_j'-i} pq;p)}
\prod_{\substack{(i,j)\in \lambda \\ \lambda_j' \neq \mu_j'}}
\frac{\theta(q^{\lambda_i-j}t^{1+\lambda_j'-i} ;p)}{\theta(q^{\mu_i+j-1}t^{-\mu_j'-i} pqa;p)}
\\ \times 
\prod_{\substack{(i,j)\in \mu \\ \lambda_j'=\mu_j'}}
\frac{\theta(q^{\lambda_i-j}t^{\lambda_j'-i} pq;p)}{\theta(q^{\mu_i+j-1}t^{1-\mu_j'-i} a;p)}
\prod_{\substack{(i,j)\in \mu \\ \lambda_j'\neq \mu_j'}}
\frac{\theta(q^{\lambda_i+j-1}t^{1-\lambda_j'-i} pq a;p)}{\theta(q^{\mu_i-j}t^{1+\mu_j'-i} ;p)}
\end{multline*}
\end{definition}
And we have the important lemma
\begin{lemma}\label{lembinonpshift}
We have
\begin{equation}\label{eqbinonp}
\binon{\lambda}{\mu}_{[pa,t];q,t;p} =
(-pqa)^{|\lambda|-|\mu|} q^{n(\lambda')-n(\mu')} t^{n(\mu)-n(\lambda)-|\lambda|}
\binon{\lambda}{\mu}_{[a,t];q,t;p} 
\end{equation}
\end{lemma}
\begin{proof}
The result follows from the basic equation of $\theta(px;p) = -\frac{1}{x} \theta(x;p)$ and some 
combinatorial arguments to simplify the expression.
%
\end{proof}

This gives the coefficients in the branching rule for interpolation functions. 

\begin{definition}
The interpolation functions are defined recursively by setting
$R_0^{*(0)}(-;q,t;p)=1$ and $R_\lambda^{*(0)}(-;q,t;p)=0$ for $\lambda \neq 0$, 
and using the branching rule
\begin{equation}\label{eqbranch}
R_{\lambda}^{*(n+1)}( \ldots, z_i, \ldots, v;a,b;q,t;p)  = \sum_{\kappa: \kappa \prec' \lambda} c_{\lambda,\kappa}
 R_{\kappa}^{*(n)}(\ldots,z_i,\ldots ;a,b;q,t;p),
\end{equation}
where
\begin{equation}\label{eqdefclk}
c_{\lambda,\kappa} = c_{\lambda,\kappa}(a,b,v;q,t;p) = \binon{\lambda}{\kappa}_{[t^{n}\frac{a}{b},t];q,t;p}
\frac{\Delta_{\lambda}^0(t^n \frac{a}{b} ~|~ t^n av, t^n \frac{a}{v}, \frac{pqa}{bt};q,t;p)}
{\Delta_{\kappa}^0(t^{n-1} \frac{a}{b} ~|~ t^n av, t^n \frac{a}{v}, \frac{pqa}{bt};q,t;p)}.
\end{equation}
\end{definition}
This definition can be seen as a generalization of the expression of Macdonald polynomials using branching rules (see \cite{Macdonald}). By expanding the sum further this defines $R_{\lambda}^{*(n)}$ as a sum over all chains 
$0=\lambda^{(0)} \prec' \lambda^{(1)} \prec' \cdots \prec' \lambda^{(n)}=\lambda$, which corresponds to a semistandard Young tableau. Thus this definition is also a direct analogue of the combinatorial definition of the Schur functions (which would correspond to the case $c_{\lambda,\kappa} = v^{|\lambda/\kappa|}$).
%

The functions defined here are identical to the ones in \cite{RainsBCn}, as follows from Corollary 4.5, Theorem 4.16 and Definition 11 from \cite{RainsBCn}. 
In \cite{CG}, Coskun and Gustafson define their well-poised Macdonald functions $W_{\lambda}$ using essentially the same formula (though it is quite non-trivial to identify the coefficients in the branching rule). The resulting equality is given by 
\begin{align*}
W_{\lambda}(\frac{z_i}{a};a^2,\frac{a}{b}) &=
\left( \frac{b^2}{q^2t^{n-1} }\right)^{|\lambda|}
q^{-2n(\lambda')} t^{2n(\lambda)}
\frac{1}{C_{\lambda}^-(t)C_{\lambda}^+(t^{n-1} \frac{aq}{bt})}
\frac{C_{\lambda}^0(t^{n})}{C_{\lambda}^0(\frac{qa}{bt}) }
 \prod_{i=1}^n \theta(q t^{n-2i} \frac{a}{b})_{2\lambda_i} R_{\lambda}^{*(n)}(z_i;a,b).
\end{align*}

Let us list some of the basic properties of these interpolation functions (all can be easily proved inductively, except for the $z_i\leftrightarrow z_j$ symmetry).
\begin{itemize}
\item We have the following negate-the-parameters symmetry
\begin{equation}\label{eqsym1}
R_{\lambda}^{*(n)}(-z_i ; -a,-b;q,t;p) = R_{\lambda}^{*(n)}(z_i;a,b;q,t;p)
\end{equation}
\item The interpolation functions $R_{\lambda}^{*(n)}(z_i;a,b)$ are $BC_n$-symmetric theta functions. Moreover they are contained in $A^{(n)}(b;p,q)$.
\item We have the ``invert all parameters'' symmetry
\begin{equation}\label{eqoverq}
R_{\lambda}^{*(n)}(z_i; \frac{1}{a}, \frac{1}{b};\frac{1}{q},\frac{1}{t};p)=
R_{\lambda}^{*(n)}(z_i;a,b;q,t;p) q^{-4n(\lambda')} t^{4n(\lambda)} \left( \frac{b^2}{a^2 q^2 t^{2(n-1)}}\right)^{|\lambda|}.
\end{equation}
\end{itemize}

Moreover we have the following 
\begin{proposition}\label{propRstarp}
The interpolation functions satisfy the following shifting formulas :
\begin{align*}
R_{\lambda}^{*(n)}(z_i;pa,b;q,t;p) & =  \left( \frac{1}{a^2 t^{n-1}} \right)^{|\lambda|} q^{-2n(\lambda')}t^{2n(\lambda)} R_{\lambda}^{*(n)}(z_i;a,b;q,t;p), \\
R_{\lambda}^{*(n)}(z_i;a,pb;q,t;p) & =  \left( \frac{b^2}{t^{n-1} q^2} \right)^{|\lambda|}
q^{-2n(\lambda')} t^{2n(\lambda)} R_{\lambda}^{*(n)}(z_i;a,b;q,t;p), \\
R_{\lambda}^{*(n)}(\sqrt{p}z_i;\sqrt{p}a,\sqrt{p}b;q,t;p) & =  \left(\frac{b}{t^{n-1} aq}\right)^{|\lambda|} 
q^{-2n(\lambda')} t^{2n(\lambda)} R_{\lambda}^{*(n)}(z_i;a,b;q,t;p). \\
\end{align*}
\end{proposition}
As a corollary we obtain
\begin{corollary}
Define the rescaled interpolation functions by
\[
Q_{\lambda}^{*(n)}(\ldots, z_i,\ldots;a,b;q,t;p) = 
R_{\lambda}^{*(n)}(\ldots, z_i,\ldots;a,b) 
\Delta_{\lambda}(\frac{t^{n-1}a}{b}~|~t^n)
\frac{C_{\lambda}^+(\frac{t^{n-1}a}{b}) C_{\lambda}^0(t^{n-1}ab)}
{C_{\lambda}^+(t^{2(n-1)}a^2) C_{\lambda}^0(\frac{pqt^{n-1}a}{b})}.
\]
Then as a function of $z_i$ we have $Q_{\lambda}^{*(n)} \in A^{(n)}(b;q,p)$.
$Q_{\lambda}^{*(n)}$ is elliptic in $a$ and $b$ (i.e. invariant under setting $a\to pa$ or $b\to pb$). Moreover it satisfies the equation
\[
Q_{\lambda}^{*(n)}(\ldots, \sqrt{p} z_i, \ldots;\sqrt{p} a,\sqrt{p}b;q,t;p)=
Q_{\lambda}^{*(n)}(\ldots, z_i, \ldots; a,b;q,t;p).
\]
\end{corollary}
The specific normalization we chose for $Q_{\lambda}^{*(n)}$ corresponds with the principal evaluation \cite[(3.37)]{RainsBCn}, which says
\begin{proposition}\label{propinterpoleval}
We have
\[
Q_{\lambda}^{*(n)}(\ldots, aq^{\lambda_i} t^{n-i},\ldots;a,b;q,t;p) = 1.
\]
\end{proposition}

It should be noted that the interpolation functions are linearly independent for generic values of the parameters. Indeed the interpolation functions 
$R_{\lambda}^{*(n)}(z_i;t_0,u_0)$ with $\lambda \subset m^n$ form a basis of the $BC_n$ symmetric theta functions $f$ such that 
\[
\prod_{i=1}^n \theta(pq z_i^{\pm 1}/u_0;q;p)_m f(\ldots,z_i,\ldots)
\]
is holomorphic.

Using the interpolation functions we define the generalized binomial coefficients (following \cite[Definition 11]{RainsBCn})
\begin{definition}\label{defbinom}
The generalized binomial coefficients are given by
\[
\binom{\lambda}{\mu}_{[a,b];q,t;p} =
\Delta_{\mu}(\frac{a}{b}~|~t^n,\frac{1}{b};q,t;p)
R_\mu^{*(n)}( \ldots, \sqrt{a} q^{\lambda_i} t^{1-i},\ldots; t^{1-n}\sqrt{a}, \frac{b}{\sqrt{a}};q,t;p).
\]
for any $n\geq l(\lambda),l(\mu)$. These binomial coefficients are independent of the choice of $\sqrt{a}$ by \eqref{eqsym1}. They are also independent of the explicit choice of $n$.
\end{definition}

 Let us give a basic symmetry of the binomial coefficients:
\begin{equation}\label{eqbinomoverq}
\binom{\lambda}{\mu}_{[1/a,1/b];1/q,1/t;p} = \binom{\lambda}{\mu}_{[a,b];q,t;p}
\end{equation}
in view of \eqref{eqoverq}. 

Moreover the binomial coefficients are elliptic in $a$ and $b$.
\begin{proposition}\label{propbinomp}
We have
\[
\binom{\lambda}{\kappa}_{[ap,b];q,t;p} = \binom{\lambda}{\kappa}_{[a,b];q,t;p} = \binom{\lambda}{\kappa}_{[a,bp];q,t;p}.
\]
\end{proposition}

\section{Biorthogonal functions}\label{secbiortho}
In this section we recall the definition and basic properties of the biorthogonal functions from \cite{RainsBCn}.

\begin{definition}\label{defbiortho}
Let $t_0$, $t_1$, $t_2$, $t_3$, $u_0$, $u_1$, $q$, and $t$ be parameters such that $t^{2(n-1)}t_0t_1t_2t_3u_0u_1=pq$. We define
\begin{equation}\label{eqdefbiortho}
\tilde R_{\lambda}^{(n)} (z_i ;t_0{}:{}t_1,t_2,t_3;u_0,u_1;q,t;p) = 
\sum_{\mu \subset \lambda}
\frac{\binom{\lambda}{\mu}_{[1/u_0u_1,1/t^{n-1}t_0u_1];q,t;p} R_{\mu}^{*(n)}(z_i ;t_0,u_0;q,t;p)}
{\Delta_{\mu}^0(t^{n-1}t_0/u_0 ~|~ t^{n-1} t_0t_1, t^{n-1} t_0t_2, t^{n-1} t_0t_3,t^{n-1}t_0u_1;q,t;p)}.
\end{equation}

\end{definition}
The normalization for this definition is chosen so that the biorthogonal functions are highly invariant under shifts of the parameters. If we set $n=1$ and $\lambda=(l)$ (i.e., $\lambda$ has one part) then this definition reduces to the 
univariate biorthogonal functions of \cite{S}. 
\begin{lemma}\label{lemRtildeabel}
As functions of $z_i$ we have $\tilde R_{\lambda}^{(n)} (\ldots,z_i,\ldots ;t_0{}:{}t_1,t_2,t_3;u_0,u_1;q,t;p) \in A^{(n)}(u_0;p,q)$.
Moreover the biorthogonal functions are elliptic in the $t_r$ and $u_r$, that is, they are invariant under multiplying these parameters with integer powers of $p$ (as long as the balancing condition remains satisfied). Finally they satisfy the equations
\[
\tilde R_{\lambda}^{(n)} (z_ip^{1/2} ;t_0p^{1/2}{}:{}t_1p^{-1/2},t_2p^{-1/2},t_3p^{-1/2};u_0p^{1/2},u_1p^{1/2};q,t;p) = 
\tilde R_{\lambda}^{(n)} (z_i ;t_0{}:{}t_1,t_2,t_3;u_0,u_1;q,t;p)
\]
and
\begin{equation}\label{eqinvq}
\tilde R_{\lambda}^{(n)} (z_i; \frac{1}{t_0}{}:{} \frac{1}{t_1},\frac{p}{t_2},\frac{p}{t_3};\frac{1}{u_0},\frac{1}{u_1};\frac{1}{q},\frac{1}{t};p) = 
\tilde R_{\lambda}^{(n)} (z_i; t_0 {}:{} t_1,t_2,t_3;u_0,u_1;q,t;p).
\end{equation}
\end{lemma}
\begin{proof}
The biorthogonal functions are written as sums of functions in the space $A^{(n)}(u_0;p,q)$, so they are in this space themselves. We can use \eqref{eqc0p} and Propositions \ref{propRstarp} and \ref{propbinomp} to show that the individual summands in the definition of the biorthogonal functions satisfy the given $p$-shift equations. The final equation follows from a direct calculation using \eqref{eqoverq} and \eqref{eqbinomoverq}.
\end{proof}

As suggested by their name, the biorthogonal functions satisfy a biorthogonality relation. There are two kinds of biorthogonality measures. For
generic parameters we have a continuous biorthogonality measure, while, if we specialize to $t^{n-1}t_0t_1=q^{-m}$ the continuous measure reduces to a discrete one. The discrete version can be obtained from the continuous biorthogonality by residue calculus.
The continuous version is \cite[Theorem 8.4 and 8.10]{Rainstrafo}, while the discrete version is \cite[Theorem 8.11]{Rainstrafo} or \cite[Theorem 5.8]{RainsBCn}
\begin{theorem}\label{thmbiorthogonality}
For any partitions $\lambda$ and $\kappa$ of length at most $n$, and for generic values of the parameters 
such that $t^{2(n-1)}t_0t_1t_2t_3u_0u_1=pq$ 
we have
\begin{multline*}
\langle \tilde R_{\lambda}^{(n)}(; t_0{}:{}t_1,t_2,t_3;u_0,u_1;q,t;p), \tilde R_{\kappa}^{(n)}(; t_0{}:{}t_1,t_2,t_3;u_1,u_0;q,t;p) \rangle_{t_0,t_1,t_2,t_3,u_0,u_1;q,t;p} \\=
\delta_{\lambda,\kappa} \frac{1}{\Delta_\lambda( \frac{1}{u_0u_1}~|~ t^n, t^{n-1}t_0t_1,t^{n-1}t_0t_2,t^{n-1}t_0t_3, \frac{1}{t^{n-1}t_0u_0}, \frac{1}{t^{n-1}t_0u_1};q,t;p)},
\end{multline*}
where 
\begin{multline*}
\langle f,g\rangle_{t_0,t_1,t_2,t_3,t_4,t_5:q,t;p} = 
\frac{(q;q)^n (p;p)^n \Gamma(t;p,q)^n}{2^n n! \prod_{j=1}^n\Gamma(t^j;p,q) \prod_{0\leq r<s\leq 5} \Gamma(t^{n-j} t_rt_s;p,q)} \\ \times 
\int_{C^n} f(\cdots,z_i,\cdots) g(\cdots,z_i,\cdots) \prod_{1\leq j<k\leq n} \frac{\Gamma(tz_j^{\pm 1} z_k^{\pm 1};p,q)}{\Gamma(z_j^{\pm 1}z_k^{\pm 1};p,q)}
\prod_{j=1}^n \frac{\prod_{r=0}^5 \Gamma(t_rz_j^{\pm 1};p,q)}{\Gamma(z_j^{\pm 2};p,q)} \frac{dz_j}{2\pi i z_j},
\end{multline*}
for parameters such that $t^{2(n-1)}t_0t_1t_2t_3t_4t_5=pq$ and functions $f\in A^{(n)}(t_4;p,q)$ and $g \in A^{(n)}(t_5;p,q)$.
Let $m_f$ be such that $f(z_i) \prod_i \Gamma(t_4z_i^{\pm 1})/\Gamma(t_4q^{-m_f} z_i^{\pm 1})$ is holomorphic, and define $m_g$ likewise for $g$. Let $\tilde t_r=t_r$ for $0\leq r\leq 3$ and $\tilde t_4 = t_4q^{-m_f}$ and $\tilde t_5 =t_5q^{-m_g}$. The contour is now taken such that $C=C^{-1}$, contains all points of the form $p^i q^j \tilde t_r$ (for $i,j\geq 0$) (and hence excludes their reciprocals), and contains the contours $p^i q^j tC$ for $i,j \geq 0$ \footnote{To be precise, $C$ should be a chain representing the described homology class.}. 

If moreover $t_0t_1=q^{-m}t^{1-n}$, and thus $t^{n-1}t_2t_3u_0u_1=pq^{m+1}$ (implying no contour of the desired shape exists) we define the inner product as 
\begin{multline*}
\langle f,g\rangle_{t_0,t_1,t_2,t_3,u_0,u_1:q,t;p}
 =
\sum_{\mu \subset m^n} 
f(t_0t^{n-i} q^{\mu_i}) g(t_0t^{n-i} q^{\mu_i}) 
\\ \times \frac{\Delta_\mu(t^{2(n-1)} t_0^2~|~ t^n, t^{n-1}t_0t_1,t^{n-1}t_0t_2, t^{n-1}t_0t_3, t^{n-1}t_0u_0,t^{n-1}t_0u_1;q,t;p)}
{\Delta_{m^n}^0(\frac{t^{n-1} t_1}{u_0} ~|~ \frac{t_1}{t_0}, \frac{pq}{u_0t_2}, \frac{pq}{u_0t_3}, \frac{pq}{u_0u_1};q,t;p)}
\end{multline*}
and have the same biorthogonality.
\end{theorem}
Note that using $\lambda=\kappa=0$ we find that the inner products are normalized such that $\langle 1,1\rangle=1$. 

The definition gives us an expansion of $\tilde R_{\lambda}^{(n)}$ in terms of the interpolation functions $R_{\mu}^{*(n)}(;t_0,u_0)$.
This is not the only basis for $A^{(n)}(u_0)$, indeed $R_{\mu}^{*(n)}(;v,u_0)$ is such a basis for all values of $v$. 
It turns out we need expansions in these different bases of interpolation functions as well, as some limits of interpolation functions are independent of $z_i$ (so we do not want to expand into those). It is convenient to first recall the 
definition of $\Omega_{\lambda/\kappa}$ from \cite{RainsBCn}.

\begin{definition}\label{defomega}
We define the coefficients $\Omega_{\lambda/\kappa}$ as 
\begin{multline*}
\Omega_{\lambda/\kappa}(a,b;v_1,v_2,v_3,v_4;q,t;p) \\ = 
\sum_{\kappa \subset \mu \subset \lambda} 
\binom{\lambda}{\mu}_{[pqa^2,pqab]} \binom{\mu}{\kappa}_{[\frac{a}{b},\frac{abpq }{v_1v_2v_3v_4}]}
\frac{\Delta_{\lambda}^0(pqa^2~|~ pqab) \Delta_\mu^0(\frac{a}{b}~|~ \frac{abpq}{v_1v_2v_3v_4})}
{\Delta_{\mu}^0(\frac{a}{b}~|~ \frac{1}{pqab}) \Delta_{\kappa}^0(\frac{v_1v_2v_3v_4}{b^2 pq}~|~ \frac{v_1v_2v_3v_4}{abpq})}
\prod_{r=1}^4 \frac{  C_{\lambda}^0(\frac{pqa}{v_r}) C_{\mu}^0(\frac{v_r}{b}) }{ C_{\mu}^0(\frac{pqa}{v_r}) C_{\kappa}^0(\frac{v_r}{b})}
\end{multline*}
\end{definition}
Univariately $\Omega_{\lambda/\kappa}$ corresponds to a very well poised series ${}_{12}V_{11}$. 
In \cite{RainsBCn} it was shown that these coefficients satisfy a Weyl group of type $D_4$ symmetry:
\begin{theorem}
$\Omega_{\lambda/\kappa}(a,b;v_1,v_2,v_3,v_4;q,t;p) $ is symmetric under permutations of $v_1$, $v_2$, $v_3$ and $v_4$ and satisfies the equation
\[
\Omega_{\lambda/\kappa}(a,b;v_1,v_2,v_3,v_4;q,t;p) =
\Omega_{\lambda/\kappa}(a,\frac{b}{v_3v_4};v_1,v_2,1/v_3,1/v_4;q,t;p) 
\]
Moreover we have
\[
\Omega_{\lambda/\kappa}(a,b;v_1,v_2,v_3,v_4;q,t;p) =
\Omega_{\lambda/\kappa}(-a,-b;-v_1,-v_2,-v_3,-v_4;q,t;p) 
\]
\end{theorem} 
There also exists an evaluation formula for the $\Omega_{\lambda/\kappa}$ if
the product of two $v$-parameters equals $abpq$.
This is the bulk difference equation \cite[Theorem 4.1]{RainsBCn}.
\begin{theorem}\label{thmomegaeval}
We have
\[
\Omega_{\lambda/\kappa}(a,b~|~v_1,v_2,x, \frac{abpq}{x}) =
\binom{\lambda}{\kappa}_{[pq a^2, \frac{pqab}{v_1v_2}]}
\frac{C_{\lambda}^0(\frac{pqab}{v_1v_2},pqav_1,pqav_2,\frac{x}{b}, \frac{pqa}{x}) C_{\kappa}^0(p^2q^2a^2)}
{C_{\lambda}^0(\frac{pqav_1v_2}{b}) C_{\kappa}^0(pqav_1,pqav_2, \frac{v_1v_2}{pqab}, \frac{x}{b},\frac{apq}{x})}.
\]
\end{theorem}
Finally we would like to observe that
\[
\Omega_{\lambda/\kappa}(a,b;v_1,v_2,v_3,v_4;q,t;p) \frac{ \Delta_{\kappa}^0(\frac{v_1v_2v_3v_4}{b^2 pq}~|~ \frac{v_1v_2v_3v_4}{abpq})}{\Delta_{\lambda}^0(pqa^2~|~ pqab) }
\prod_{r=1}^4 \frac{  C_{\kappa}^0(\frac{v_r}{b})}{  C_{\lambda}^0(\frac{pqa}{v_r})  }
\]
is elliptic in $a$, $b$, $v_r$, i.e invariant under $a\to ap$, $b\to bp$, $v_1\to v_1p$, etc.

We can get other expansions for the biorthogonal functions, by using the definition and then expanding $R^{*(n)}(;t_0,u_0)$ in $R^{*(n)}(;v,u_0)$ (for any $v$ of our choice) using \cite[Corollary 4.14]{RainsBCn}. This gives the equation
\begin{multline}\label{eqbiorthousingomega}
\tilde R_{\lambda}^{(n)}(;t_0{}:{}t_1,t_2,t_3;u_0,u_1;q,t;p) = 
\frac{C_{\lambda}^0(\frac{pqt^{n-1}t_0}{u_0})}{C_{\lambda}^0(\frac{pq}{vu_0}, t^{n-1}t_0t_1,t^{n-1}t_0t_2,t^{n-1}t_0t_3, \frac{1}{t^{n-1}t_0u_1})} \\ \times 
\sum_{\nu \subset \lambda}
\Omega_{\lambda/\nu}(s, \frac{u_0s}{t^{n-1}t_0 }; 
\frac{pqs}{t_0t_1t^{n-1}}, \frac{pqs}{t_0t_2t^{n-1}}, \frac{pqs}{t_0t_3t^{n-1}}, u_0vs;q,t;p)
\\ \times 
C_\nu^0( \frac{pq}{u_0t_0}, \frac{pq}{u_0t_1}, \frac{pq}{u_0t_2}, \frac{pq}{u_0t_3})
R_{\nu}^{*(n)}(;v,u_0),
\end{multline}
where $s=1/\sqrt{pq u_0u_1}$. Note that the summation over $\nu$ (without the $C_{\lambda}^0$ terms in front) is permutation 
symmetric in $t_0$, $t_1$, $t_2$ and $t_3$ (using the symmetries of $\Omega_{\lambda/\nu}$); in particular this shows that the biorthogonal functions themselves are symmetric under the exchange of $t_0$ and $t_1$ (or just any permutation of $(t_0,t_1,t_2,t_3)$), as long as we multiply by the right product of $C_{\lambda}^0$'s. The current normalization is such that 
\begin{equation}\label{eqnormR}
\tilde R_{\lambda}^{(n)} (t^{n-i} t_0;t_0:t_1,t_2,t_3;u_0,u_1) = 1.
\end{equation}
Together with the permutation symmetry in the $t_r$ this also gives us evaluations for $\tilde R_{\lambda}^{(n)}(t^{n-i}t_r)$ for $r=1,2,3$. 

For some values of $v$ we can evaluate the $\Omega_{\lambda/\nu}$ factors in this expansion, not surprisingly these
are $v=t_r$ ($r=0,1,2,3$), which lead to the original expansion (if $v=t_0$), and versions of that one using the symmetry in 
the $t_r$.

We would like to mention a few results on these biorthogonal functions, which are essentially algebraic equations in our space of formal power series in $p$. Thus it should be relatively straightforward to obtain limits of these equations. However we have not yet completed the full study of obtaining these limits.

We begin with the following evaluation duality \cite[Theorem 5.4]{RainsBCn}, which is the analog of Macdonald's evaluation duality conjecture. Once we have obtained the limits of the biorthogonal functions the relation will reduce to duality relations for our limiting biorthogonal functions. 
\begin{equation}\label{eqduality}
\tilde R_{\lambda}^{(n)}(t_0t^{n-i}q^{\kappa_i};t_0:t_1,t_2,t_3;u_0,u_1) = 
\tilde R_{\kappa}^{(n)}(\hat t_0t^{n-i}q^{\lambda_i};\hat t_0:\hat t_1,\hat t_2,\hat t_3;\hat u_0,\hat u_1),
\end{equation}
where the new parameters are given by 
\[
\hat t_0 = \sqrt{t_0t_1t_2t_3/pq}, \qquad 
\hat t_0\hat t_r =t_0t_r, \quad (r=1,2,3), \qquad 
\hat t_0/\hat u_r = t_0/u_r, \quad (r=0,1).
\]
Notice that we need the valuation of $z$ to equal the valuation of $t_0$ (i.e. $\zeta=\alpha_0$) for the limit to work (though by permutation symmetry in the $t_r$'s and the $z\leftrightarrow 1/z$ symmetry we have a few more choices). Also observe that the case $\kappa=0$ is exactly the normalization equation.

Let us define a difference operator  \cite[(6.18)]{Rainstrafo}
\begin{align*}
D_q^{(n)}(v;t,p) f(z_i) := \sum_{\sigma \in \{\pm 1\}^n}
\frac{\prod_{r=0}^3 \theta(v_rz_i^{\sigma_i};p)}{\theta(z_i^{2\sigma_i};p)}
\prod_{1\leq i<j\leq n} \frac{\theta(tz_i^{\sigma_i} z_j^{\sigma_j};p)}{\theta(z_i^{\sigma_i}z_j^{\sigma_j};p)} f(q^{\sigma_i/2}z_i).
\end{align*}
It is easily shown that if $t^{n-1}v_0v_1v_2v_3=p$ this difference operators maps the space of  $BC_n$-symmetric abelian functions to itself. We can rescale this operator to
\[
D_q^{(n)}(v_0,v_1,v_2;t,p) f = \frac{D_q^{(n)}(v_0,v_1,v_2, p/t^{n-1}v_0v_1v_2;t,p)f}{\prod_{i=1}^{n} 
\theta(t^{n-i} v_0v_1,t^{n-i} v_0v_2,t^{n-i} v_1v_2;p)}
\]
so that we obtain the difference equation
\[
D_q^{(n)}(u_0,t_0,t_1;t,p) \tilde R_{\lambda}^{(n)}(\cdot; q^{1/2} t_0: q^{1/2}t_1,q^{-1/2}t_2,q^{-1/2}t_3;q^{1/2}u_0,q^{-1/2}u_1) = 
\tilde R_{\lambda}^{(n)}(\cdot;t_0:t_1,t_2,t_3;u_0,u_1).
\]

In the same vein we can define
\[
D_q^{-(n)}(u_0;t,p) = D_q^{(n)}(u_0,qu_0,p/u_0, 1/t^{n-1}u_0q;t,p)
\]
and we set 
\begin{align*}
D_q^{+(n)}(v_0:v_1:v_2,v_3,v_4;t,p) f(z_i) &= 
\prod_{i=1}^n \frac{\theta(pqt^{n-i}v_1/v_0;p)}{\prod_{2\leq r\leq 5} \theta(v_rt^{n-i}v_1;p)}
\\ & \qquad \times \sum_{\sigma \in \{\pm 1\}^n} \prod_{i=1}^n \frac{\prod_{r=1}^5 \theta(v_r z_i^{\sigma_i};p)}{\theta(pqz_i^{\sigma_i}/v_0, z_i^{2\sigma_i};p)}
\prod_{1\leq i<j\leq n} \frac{\theta(tz_i^{\sigma_i}z_j^{\sigma_j};p)}{\theta(z_i^{\sigma_i}z_j^{\sigma_j};p)}f(q^{\sigma_i/2}z_i),
\end{align*}
where $v_5$ is determined by the equation $t^{n-1}\prod_{r=0}^5 v_r= p^2q$.
Acting on the biorthogonal functions these operators give the equations
\[
D_q^{+(n)}(u_0:t_0:t_1,t_2,t_3) \tilde R_{\lambda}^{(n)}(\cdot;
q^{1/2}t_0:q^{1/2}t_1,q^{1/2}t_2,q^{1/2}t_3;q^{-1/2}u_0,q^{-3/2}u_1)
= \tilde R_{\lambda+ 1^n}^{(n)}(\cdot; t_0:t_1,t_2,t_3;u_0,u_1)
\]
and
\begin{align*}
D_q^{-(n)}&(u_0)  \tilde R_{\lambda+1^n}^{(n)}(\cdot;q^{-1/2}t_0:q^{-1/2}t_1,q^{-1/2}t_2,q^{-1/2}t_3;q^{3/2}u_0,q^{1/2}u_1) \\ & = 
\prod_{i=1}^n \frac{\prod_{r=0}^3 \theta(t^{n-i}u_0t_r;p) \theta(\frac{qu_0}{t_0}t^{i-n},\frac{t^{n-i}}{u_1t_0}, \frac{u_0}{t_0}t^{i-n} ,\frac{t^{n-i}q^{\lambda_i-1}}{u_0u_1}, t^{i-n}q^{-\lambda_i-1};p)}
{\prod_{r=1}^3 \theta(t^{n-i}t_0t_r/q;p) \theta(
\frac{t^{i+1-2n}}{t_0u_1}, \frac{1}{u_1t_0} t^{n-i}q^{\lambda_i},  \frac{u_0}{t_0}q^{-\lambda_i}t^{i-n} ;p)}
\tilde R_{\lambda}^{(n)}(\cdot; t_0:t_1,t_2,t_3;u_0,u_1).
\end{align*}

\section{Limits of interpolation functions and binomial coefficients}\label{secliminterpolation}
In this section we discuss the limits of interpolation functions and binomial coefficients. This section only contains the methodology and a description of the results. The explicit calculations and expressions for the explicit limits we obtain are relegated to Appendix \ref{apA}. 

First we wish to remark that the interpolation functions, and therefore also the binomial coefficients, are elements of the space $M(z_i,a,b,q,t)$, respectively $M(a,b,q,t)$, as explained in Section \ref{secnot}. In particular finding their limits consists of determining the leading coefficients, and convergence is then automatically uniform on compacta in $\mathbb{C}^*$ outside the zero set of a polynomial. 

Interpolation functions are defined using the branching rule \eqref{eqbranch}, which is 
\begin{equation*}
R_{\lambda}^{*(n+1)}( \ldots, z_i, \ldots, v;a,b;q,t;p)  = \sum_{\kappa: \kappa \prec' \lambda} c_{\lambda,\kappa}
R_{\kappa}^{*(n)}(\ldots,z_i,\ldots ;a,b;q,t;p),
\end{equation*}
so in order to obtain their limit we first need to find the limit of the coefficients $c_{\lambda,\kappa}$, which were given by \eqref{eqdefclk}. As we already know the limits of the $C_{\lambda}^0$'s and $C_{\kappa}^0$'s appearing, we can write down this limit as soon as we know the limit of $\binon{\lambda}{\kappa}_{[a,t];q,t;p}$. As $\binon{\lambda}{\kappa}$ is given as a product of theta functions, that limit is immediate.

It turns out that the valuation of $c_{\lambda,\kappa}(vp^{\zeta};ap^{\alpha},bp^{\beta})$ 
is always of the form $x (|\lambda|-|\kappa|)$, where $x=x(\zeta,\alpha,\beta)$ is an explicit piecewise linear function. In particular, an immediate induction shows that the valuation of $R_{\lambda}^{*(n)}(z_ip^{\zeta};ap^{\alpha},bp^{\beta})$ equals $x |\lambda|$ (assuming no cancellation occurs), and the limit of $R_{\lambda}^{*(n)}$ can be defined recursively using a branching rule \eqref{eqbranch}, where we use the limits of $c_{\lambda,\kappa}$ instead of $c_{\lambda,\kappa}$ themselves. 

So far, we have discussed the limit of $R_{\lambda}^{*(n)}(z_ip^{\zeta};ap^{\alpha},bp^{\beta})$ for a given $(\alpha,\beta,\zeta)\in \mathbb{R}^3$. However, as there are infinitely many such vectors, we now want to bring some order in these different limits, and show that in fact there are only finitely many essentially different limits. 

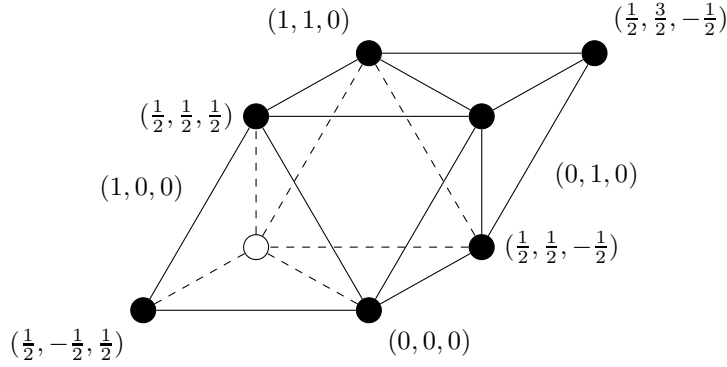
\begin{figure}
\begin{center}

\begin{tikzpicture}[scale=0.6,vol/.style={circle, fill, draw=black},leeg/.style={circle, minimum size=4pt, draw=black}]

\node[label=below left:{$(\frac12,-\frac12,\frac12)$}] (pmp) at (0,0) [vol] {};
\node[label=below right:{$(0,0,0)$}] (nnn) at (5,0) [vol] {};
\node[label=left:{$(\frac12,\frac12,\frac12)$}] (ppp) at (2.5,4.3) [vol] {} ;
\node[label=above left:{$(1,1,0)$}] (een) at (5,5.7) [vol] {} ;
\node[label=right:{$(\frac12,\frac12,-\frac12)$}] (ppm) at (7.5,1.4) [vol] {} ;
\node[label=above right:{$(\frac12,\frac32,-\frac12)$}] (pdm) at (10,5.7) [vol] {} ;
\begin{scope}[label distance=8mm]
\node[label=150:{$(1,0,0)$}] (enn) at (2.5,1.4) [leeg] {} ;
\node[label=330:{$(0,1,0)$}] (nen) at (7.5,4.3) [vol] {} ;
\end{scope}



\draw (pmp) -- (nnn) -- (ppm) -- (pdm) -- (een) -- (ppp) -- (pmp) ;
\draw (ppp) -- (nnn) ;
\draw (ppm) -- (nen) -- (een) ;
\draw (ppp) -- (nen) -- (nnn) ;
\draw (nen) -- (pdm) ;
\draw[dashed] (pmp) -- (enn) -- (ppp) ;
\draw[dashed] (nnn) -- (enn) -- (een) -- (ppm) -- (enn) ;

\end{tikzpicture}

\caption{The fundamental domain including its partition in a octahedron and two tetrahedra. The points are labeled by the values of $(\alpha,\beta,\zeta)$. 
The difference between solid lines and dots on the one hand, and dashed lines and open dots on the other, is that the former would be visible were the parallelepiped solid, while the latter would be invisible.}\label{fig1}

\end{center}

\end{figure}

Recall the $p$-shifts in the arguments of $R_{\lambda}^{*(n)}$ as given in Proposition \ref{propRstarp}, and the fact that $R_{\lambda}^{*(n)}$ is $p$-abelian in the $z_i$-parameters. As a consequence 
we can rewrite $R_{\lambda}^{*(n)}(z_ip^{\zeta};ap^{\alpha},bp^{\beta})$ for any $(\alpha,\beta,\zeta)$, as an explicit power of $a$, $b$, $q$, $t$ and $p$ times an interpolation function with parameters $(\hat \alpha,\hat \beta,\hat \zeta)$ in some fundamental domain of the group $G$ of shifts along the lattice generated by 
$(1,0,0)$, $(0,1,0)$ and $(\frac12,\frac12,\frac12)$. In particular, up to a simple multiplicative factor, all the limits of the interpolation functions, are limits of interpolation functions with $(\alpha,\beta,\zeta)$ in such a fundamental domain. Let us fix this fundamental domain as the parallelepiped $0\leq \alpha\pm \zeta,\beta+\zeta\leq 1$. 

Now recall that the leading coefficient of $C_{\lambda}^0(p^{\chi}x)$ only depends on $\chi$ via the component containing $\chi$ in the decomposition $\mathbb{R}=\mathbb{Z} \cup \bigcup_{n\in \mathbb{Z}} (n,n+1)$. The same property holds (by inspection) for $\binon{\lambda}{\kappa}_{[ap^{\chi},t]}$. Now we immediately see that the leading coefficient of $c_{\lambda,\kappa}$ only depends on where the numbers $(\alpha-\beta,\alpha+\zeta,\alpha-\zeta,\beta+\zeta,\beta-\zeta)$ are with respect to this partition of $\mathbb{R}$. In particular, if we cut $\mathbb{R}^3$ by the hyperplanes $\alpha-\beta\in \mathbb{Z}$, $\alpha\pm \zeta \in \mathbb{Z}$ and $\beta\pm \zeta \in \mathbb{Z}$, the leading coefficient of $R_{\lambda}^{*(n)}(z_ip^{\zeta};ap^{\alpha},bp^{\beta})$ only depends on where  the vector $(\alpha,\beta,\zeta)$ is located with respect to these hyperplanes. 

These hyperplanes cut space in a tessellation of square pyramids and tetrahedra, where the square pyramids pair up into octahedra, such that around each octahedron we only find tetrahedra and vice versa. Our fundamental domain consists of one octahedron (i.e. two square pyramids) and two of its neighboring tetrahedra, see Figure \ref{fig1}. The limit now only depends on which open polytope (i.e. vertex, edge without vertex, triangle/square without edges or interior of tetrahedron/pyramid) in this tessellation contains the vector $(\alpha,\beta,\zeta)$. 
Of course two polytopes related by a shift from the translation group $G$ will provide the same limits (up to an explicit power of $a$, $b$, $q$ and $t$). 

So in principle we have 1 limit associated to a vertex (the group $G$ acts transitively on the vertices of this partition of $\mathbb{R}^3$), 6 to an edge, 8 to a triangle, 1 to a square, 2 to a tetrahedron, and 2 to a square pyramid. We can use symmetries of the interpolation function we have not yet used to cut back this number somewhat. 

First of all the $z_i\to 1/z_i$ symmetry allows us to identify limits corresponding to vectors
$(\alpha,\beta,\zeta)$ and $(\alpha, \beta,-\zeta)$, i.e. limits which are related by a reflection in the plane $\zeta=0$. Secondly the $(a,b;q,t) \to (1/a,1/b;1/q,1/t)$ symmetry \eqref{eqoverq}, allows us to relate the limits at $(\alpha,\beta,\zeta)$ and $(-\alpha,-\beta,\zeta)$, i.e. limits related by a rotation of $180^{\circ}$ around the $\zeta$-axis. In practice we will often prefer not to use the last symmetry, as it breaks the condition $|q|<1$, which is necessary in the measures (though not for defining the interpolation functions).

Writing down the limits explicitly we find that the limits associated to vectors inside octahedra (i.e. either the square, or one of the two pyramids) are independent of $z_i$ (and therefore not particularly suited as a basis of functions to expand other functions in). However the limit associated to a vector in a tetrahedron becomes the Macdonald polynomials (in either $z_i$ or $1/z_i$), and thus in particular an independent set of symmetric functions of the $z_i$. Now recall the iterated limit property, Proposition \ref{propiteratedlim}. This implies that if we have a limit $L$ associated to the polytope $P$, we can take further limits of $L$ (by sending some parameters to 0 or infinity) obtaining the limit associated to any polytope which contains $P$ in its boundary. All limits not associated to the interior of the octahedron have a limit to the limit associated to the interior of one of the  tetrahedra. In particular all limits not associated to the interior of the octahedron must form an independent set of functions of the $z_i$.

While these families of limiting interpolation functions all form independent families of functions of the $z_i$, they do not all span the same space. Indeed, as mentioned the interiors of the tetrahedra correspond to permutation symmetric polynomials. 
If the vector $(\alpha,\beta,\zeta)$ is on a plane $\beta+\zeta\in \mathbb{Z}$, 
respectively $\beta-\zeta\in \mathbb{Z}$, 
then the functions have poles at $z_i \in b^{-1} q^{\mathbb{Z}_{\geq 1}}$, 
respectively  $z_i \in b q^{\mathbb{Z}_{\leq -1}}$, as can be seen by inspection. 
Thus whenever $\beta + \zeta \in \mathbb{Z}$ or $\beta-\zeta \in \mathbb{Z}$ the limits are rational functions of the $z_i$. Moreover we get limits which span spaces of $BC_n$-symmetric functions (i.e. invariant under $z_i \to 1/z_i$) or functions which lack this kind of symmetry, and other distinctions.

To obtain the limits for the binomial coefficients $\binom{\lambda}{\mu}_{[a,b];q,t;p}$ (Definition \ref{defbinom}) we can now just plug in the known limits of the interpolation functions in their definition. 
The result can be described in much the same way as the limits of the interpolation functions. However, we consider the limit of $\binom{\lambda}{\mu}_{[ap^{\alpha},bp^{\beta}];q,t;p}$, and see that the limit depends only on which part of the plane $\mathbb{R}^2$ contains $(\alpha,\beta)$ if we cut the plane by the lines 
$\alpha \in \mathbb{Z}$, $\beta\in \mathbb{Z}$ and $\alpha-\beta\in \mathbb{Z}$. It turns out the limits are quite a bit prettier as there are no constant factors in the $p$-shifts (Proposition \ref{propbinomp}) of the binomial constants, which implies that the valuation of the binomials is always 0.

Let us end this section by formulating the main points in a proposition 
\begin{proposition}\label{propliminterpolation} We have the following results about the limits of the interpolation functions.
\begin{enumerate}
\item For each vector $(\alpha,\beta,\zeta)\in \mathbb{R}^3$ there exists a scale $x(\alpha,\beta,\zeta)$ such that
\begin{align*}
val(R_{\lambda}^{*(n)}(z_ip^{\zeta};ap^{\alpha},bp^{\beta};q,t;p)) &= |\lambda| x(\alpha,\beta,\zeta),  \\ 
lc(R_{\lambda}^{*(n)}(z_ip^{\zeta};ap^{\alpha},bp^{\beta};q,t;p))& =: R_{\lambda,(\alpha,\beta,\zeta)}^{*(n)}(z_i;a,b;q,t)
\end{align*}
 $R_{\lambda,(\alpha,\beta,\zeta)}^{*(n)}$ is a rational function of the $z_i\in \mathbb{C}^{*}$ with poles in at most $z_i^{\pm 1} \in bq^{\mathbb{Z}_{\leq -1}}$.
\item The translation group $G$ generated by shifts in the directions $(1,0,0)$, $(0,1,0)$, and $(\frac12,\frac12,\frac12)$ when acting on the vector $(\alpha,\beta,\zeta)$, leaves 
$R_{\lambda,(\alpha,\beta,\zeta)}^{*(n)}$ invariant up to an integer power of $a$, $b$, $q$, and $t$.
\item The function $R_{\lambda,(\alpha,\beta,\zeta)}^{*(n)}$ depends on $(\alpha,\beta,\zeta)$, only through the location of this vector with respect to the hyperplanes $\alpha-\beta,\alpha\pm \zeta, \beta\pm \zeta\in \mathbb{Z}$.
\item If $(\alpha,\beta,\zeta)$ is not in the $G$-orbit of the interior of the octahedron (from Figure \ref{fig1}) then the $R_{\lambda,(\alpha,\beta,\zeta)}^{*(n)}$ form an independent family of functions of the $z_i$, that is, there does not exist a non-trivial equation of the form
\[
\sum_{\lambda} d_{\lambda}(a,b,q,t)R_{\lambda,(\alpha,\beta,\zeta)}^{*(n)}(z_i;a,b;q,t)=0,
\]
where the sum is finite.
\item For any vector $(\alpha,\beta)\in \mathbb{R}^2$ we have 
\[
val\left( \binom{\lambda}{\mu}_{[ap^{\alpha},bp^{\beta}];q,t;p}\right)=0, \qquad
lc\left(\binom{\lambda}{\mu}_{[ap^{\alpha},bp^{\beta}];q,t;p}\right) =:\binom{\lambda}{\mu}_{(\alpha,\beta);[a,b];q,t}
\]
The translation group $H$ generated by shifts in the directions $(1,0)$, and  $(0,1)$ when acting on the vector $(\alpha,\beta)$, leaves $\binom{\lambda}{\mu}_{ (\alpha,\beta);[a,b]}$ invariant. The limit $\binom{\lambda}{\mu}_{ (\alpha,\beta);[a,b]}$ depends on $(\alpha,\beta)$ only through the 
location of this vector with respect to the hyperplanes $\alpha,\beta,\alpha-\beta\in \mathbb{Z}$.
\end{enumerate}

\end{proposition}

\section{Limits of the biorthogonal functions}\label{seclimbiortho}
The goal of this section is to show two things. First of all we want to show that the limit of the biorthogonal functions are  $z$-dependent if and only if the limit of the univariate biorthogonal functions are $z$-dependent. Moreover the limit of the biorthogonal functions with $\lambda \subset m^n$ form a basis of the appropriate limit space, when this occurs. Secondly we want to show that there always exists an expansion (of the form \eqref{eqbiorthousingomega}, with $\Omega_{\lambda/\nu}$ replaced by its definition, Definition \ref{defomega}) for which we can obtain the limit of the biorthogonal functions by simply interchanging limit and sum.

  In order to study the limits of the biorthogonal functions, it is necessary to first consider the limits of $\Omega_{\lambda/\kappa}$.
Indeed, interchanging limit and sum in the defining expansion \eqref{eqdefbiortho} of $\tilde R_{\lambda}^{(n)}$ does not always work. In that sum we expand $\tilde R_{\lambda}^{(n)}$ in $R_{\mu}^{*(n)}(;t_0,u_0)$ (for $\mu \subset \lambda$), and we have seen in the previous section that for some ways of scaling the $z_i$, $t_0$ and $u_0$ with $p$ the limit of these interpolation functions will not depend on $z$. This might mean that the limit is itself $z$-independent, but it could also be that the valuation of $\tilde R_{\lambda}^{(n)}$ is more than the valuation of the summands in \eqref{eqdefbiortho}. By using equation \eqref{eqbiorthousingomega} instead we can expand $\tilde R_{\lambda}^{(n)}$ in $R_{\lambda}^{*(n)}(;v,u_0)$ for a $v$ of our choice, and in particular we can choose $v$ so that at least the interpolation functions have a proper limit, and their limits form an linearly independent family of  $z$-dependent functions.

\subsection{Limits of Omega} 
In this subsection we show that at least one of the eight different sum expressions we have for $\Omega_{\lambda/\kappa}$ (using its symmetries) is such that we can determine its valuation as the minimum of the valuations of the summands (in the expansion in powers of $p$). In particular this implies we can always determine the valuation of $\Omega_{\lambda/\kappa}$. 

Recall that the valuation of a sum of a finite number of terms is at least the minimum of the valuations of the individual summands, and equals the minimum if the minimum is attained at most once. Moreover, if the valuation of the sum equals the minimum of the valuations of the summands, then the leading coefficient of the sum equals the sum of the leading coefficients of summands with minimal valuation. On the other hand, if the valuation of multiple summands is the same, the leading coefficients of those summands could cancel each other, and we would have a priori no good way of expressing either the valuation or the leading coefficient of the sum in terms of leading coefficients and valuations of summands.
 
Let us first consider the case where we can determine the valuation and leading coefficient of $\Omega_{\lambda/\kappa}$ because there is a unique term in its sum expression with minimal valuation. 
\begin{proposition}\label{proplimomega}
Define the function
\[
f(\alpha,\beta;\gamma_1,\gamma_2,\gamma_3,\gamma_4) = 
g(\alpha+\beta-\sum_{r=1}^4 \gamma_r) + g(2\alpha)-g(-\alpha-\beta)-g(\sum_{r=1}^4 \gamma_r-2\beta)+
\sum_{r=1}^4 g(\gamma_r-\beta)  - g(\alpha-\gamma_r),
\]
where $g(x) = \{x\}(1-\{x\})$ (with $\{x\}$ the fractional part of $x$).
Let $a=\tilde a p^{\alpha}$, $b=\tilde b p^{\beta}$ and $v_r=\tilde v_r p^{\gamma_r}$ (where $\tilde a$, $\tilde b$ and $\tilde v_r$ are all $p$-independent).
\begin{itemize}
\item If $f(\alpha,\beta;\gamma)>0$ we have
\[
lc(\Omega_{\lambda/\kappa}(a, b;v_1,v_2,v_3,v_4 ;q,t;p)) = 
lc\left(\prod_{r=1}^4 \frac{ C_{\lambda}^0(\frac{v_r}{b}) }{  C_{\kappa}^0(\frac{v_r}{b})}
\frac{ \Delta_\lambda^0(\frac{a}{b}~|~ \frac{abpq}{v_1v_2v_3v_4})}
{\Delta_{\kappa}^0(\frac{v_1v_2v_3v_4}{b^2 pq}~|~ \frac{v_1v_2v_3v_4}{abpq})}
\frac{C_{\lambda}^+(pqa^2)}{C_{\lambda}^+(\frac{a}{b})}
 \binom{\lambda}{\kappa}_{[\frac{a}{b},\frac{abpq }{v_1v_2v_3v_4}]} \right)
\]
and likewise with the valuations, i.e.
\[
val(\Omega_{\lambda/\kappa}(a, b;v_1,v_2,v_3,v_4 ;q,t;p)) = 
val\left(\prod_{r=1}^4 \frac{ C_{\lambda}^0(\frac{v_r}{b}) }{  C_{\kappa}^0(\frac{v_r}{b})}
\frac{ \Delta_\lambda^0(\frac{a}{b}~|~ \frac{abpq}{v_1v_2v_3v_4})}
{\Delta_{\kappa}^0(\frac{v_1v_2v_3v_4}{b^2 pq}~|~ \frac{v_1v_2v_3v_4}{abpq})}
\frac{C_{\lambda}^+(pqa^2)}{C_{\lambda}^+(\frac{a}{b})}
 \binom{\lambda}{\kappa}_{[\frac{a}{b},\frac{abpq }{v_1v_2v_3v_4}]} \right).
\]
\item If $f(\alpha,\beta;\gamma)<0$ we have 
\[
lc(\Omega_{\lambda/\kappa}(a, b;v_1,v_2,v_3,v_4;q,t;p)) = 
lc\left(\prod_{r=1}^4 \frac{  C_{\lambda}^0(\frac{pqa}{v_r}) }{ C_{\kappa}^0(\frac{pqa}{v_r})}
\frac{\Delta_{\lambda}^0(pqa^2~|~ pqab)}
{\Delta_{\kappa}^0(\frac{a}{b}~|~ \frac{1}{pqab}) }
\frac{C_{\kappa}^+(\frac{a}{b})}{C_{\kappa}^+(\frac{v_1v_2v_3v_4}{pqb^2})}
\binom{\lambda}{\kappa}_{[pqa^2,pqab]} \right)
\]
and likewise with the valuations.
\end{itemize}
\end{proposition}
\begin{proof}
Recall that \eqref{eqlimc} implies 
\[
val(C_{\mu}^0(p^{\alpha}x)) = \frac12 |\mu| (  \{\alpha\}(\{\alpha\}-1)- \alpha(\alpha-1)).
\]
Also recall from the previous section that $val\left( \binom{\lambda}{\nu}_{[a,b]} \right) =0$. 
Thus for the individual summands in Definition \ref{defomega} of $\Omega_{\lambda/\kappa}$ we find
\begin{multline*}
val\left(\prod_{r=1}^4 \frac{  C_{\lambda}^0(\frac{pqa}{v_r}) C_{\mu}^0(\frac{v_r}{b}) }{ C_{\mu}^0(\frac{pqa}{v_r}) C_{\kappa}^0(\frac{v_r}{b})}
\frac{\Delta_{\lambda}^0(pqa^2~|~ pqab) \Delta_\mu^0(\frac{a}{b}~|~ \frac{abpq}{v_1v_2v_3v_4})}
{\Delta_{\mu}^0(\frac{a}{b}~|~ \frac{1}{pqab}) \Delta_{\kappa}^0(\frac{v_1v_2v_3v_4}{b^2 pq}~|~ \frac{v_1v_2v_3v_4}{abpq})}
\binom{\lambda}{\mu}_{[pqa^2,pqab]} \binom{\mu}{\kappa}_{[\frac{a}{b},\frac{abpq }{v_1v_2v_3v_4}]} \right)
\\=
-\frac12 |\mu| f(\alpha,\beta,\gamma) + 
 val\left( \Delta_{\lambda}^0(pqa^2~|~ pqab)\prod_{r=1}^4 C_{\lambda}^0(\frac{pqa}{v_r}) \right)
 - val\left(\Delta_{\kappa}^0(\frac{v_1v_2v_3v_4}{b^2 pq}~|~ \frac{v_1v_2v_3v_4}{abpq}) \prod_{r=1}^4 C_{\kappa}^0(\frac{v_r}{b}) \right).
\end{multline*}

In particular we see that if $f(\alpha,\beta,\gamma)\neq 0$ the valuation of the summands is $|\mu|$-dependent. If $f>0$ it is minimized if $|\mu|$ is largest, thus only at $\mu=\lambda$. On the other hand if $f<0$ it is minimized if $\mu=\kappa$. Together with the discussion before the proposition on how the valuations of summands correspond to valuations of the sum, this finishes the proof.

We can simplify the summands at $\mu=\lambda$ and $\mu=\kappa$ by using the explicit value of $\binom{\lambda}{\lambda}_{[a,b];q,t;p}$ from \cite[(4.8)]{RainsBCn}.
\end{proof}

In case the function $f$ in the above proposition vanishes we still don't know the valuation of $\Omega_{\lambda/\kappa}$ or its leading coefficient. However we can revert to Corollary \ref{corsumequalval} of the iterated limit theorem to tell us the valuation and leading coefficient of $\Omega_{\lambda/\kappa}$ in some more situations.
\begin{proposition}
We use the notation of the previous proposition. If $f(\alpha,\beta;\gamma)=0$, and there exists a vector $x=(x_a,x_b,x_{\gamma_1},\ldots,x_{\gamma_4})$ such that
there exists a $\delta>0$ such that for any $\epsilon$ with $\delta>\epsilon>0$ we have $f(\alpha+\epsilon x_a, \beta+\epsilon x_b; \gamma_1+\epsilon x_{\gamma_1}, \ldots, \gamma_4+\epsilon x_{\gamma_4})\neq 0$ then we have
\begin{multline*}
val(\Omega_{\lambda/\kappa}(a, b;v_1,v_2,v_3,v_4;q,t;p))
\\ = val\left( \prod_{r=1}^4 \frac{  C_{\lambda}^0(\frac{pqa}{v_r}) C_{\mu}^0(\frac{v_r}{b}) }{ C_{\mu}^0(\frac{pqa}{v_r}) C_{\kappa}^0(\frac{v_r}{b})}
\frac{\Delta_{\lambda}^0(pqa^2~|~ pqab) \Delta_\mu^0(\frac{a}{b}~|~ \frac{abpq}{v_1v_2v_3v_4})}
{\Delta_{\mu}^0(\frac{a}{b}~|~ \frac{1}{pqab}) \Delta_{\kappa}^0(\frac{v_1v_2v_3v_4}{b^2 pq}~|~ \frac{v_1v_2v_3v_4}{abpq})}
\binom{\lambda}{\mu}_{[pqa^2,pqab]} \binom{\mu}{\kappa}_{[\frac{a}{b},\frac{abpq }{v_1v_2v_3v_4}]} \right)
\end{multline*}
for any $\kappa \subset \mu \subset \lambda$.
Moreover we have
\begin{multline*}
lc(\Omega_{\lambda/\kappa}) \\= 
\sum_{\kappa \subset \mu \subset \lambda}
lc\left( \prod_{r=1}^4 \frac{  C_{\lambda}^0(\frac{pqa}{v_r}) C_{\mu}^0(\frac{v_r}{b}) }{ C_{\mu}^0(\frac{pqa}{v_r}) C_{\kappa}^0(\frac{v_r}{b})}
\frac{\Delta_{\lambda}^0(pqa^2~|~ pqab) \Delta_\mu^0(\frac{a}{b}~|~ \frac{abpq}{v_1v_2v_3v_4})}
{\Delta_{\mu}^0(\frac{a}{b}~|~ \frac{1}{pqab}) \Delta_{\kappa}^0(\frac{v_1v_2v_3v_4}{b^2 pq}~|~ \frac{v_1v_2v_3v_4}{abpq})}
\binom{\lambda}{\mu}_{[pqa^2,pqab]} \binom{\mu}{\kappa}_{[\frac{a}{b},\frac{abpq }{v_1v_2v_3v_4}]} \right)
\end{multline*}
\end{proposition}
Note that the different expressions for the valuations above do not depend on $\mu$ precisely because $f$ vanishes.
\begin{proof}
This is a direct consequence of Corollary \ref{corsumequalval} extended to an arbitrary finite number of summands. If
$f(\alpha+\epsilon x_a, \beta+\epsilon x_b; \gamma_1+\epsilon x_{\gamma_1}, \ldots, \gamma_4+\epsilon x_{\gamma_4})> 0$ (for all small positive $\epsilon$) then the summand with $\mu=\lambda$ will dominate all other terms in the iterated limit, whereas if $f(\alpha+\epsilon x_a, \beta+\epsilon x_b; \gamma_1+\epsilon x_{\gamma_1}, \ldots, \gamma_4+\epsilon x_{\gamma_4})< 0$ the summand with $\mu=\kappa$ will dominate all the others.
\end{proof}

Notice that $g(x)-x(1-x)$ is a piecewise linear of $x$. Now we can see that $f$ is piecewise linear, by replacing all instances of $g$ in the definition of $f$ by $x(1-x)$ and observing that the result vanishes. Hence, the points of space on which $f$ is identically zero in a neighborhood around that point are some polytopes. To obtain a proper limit for $\Omega_{\lambda/\kappa}$ on those sets
we can use the $W(D_4)$ symmetry. Indeed, while the sum $\Omega_{\lambda/\kappa}$ is invariant under this symmetry the summands are not, and in particular the function $f$ which controls whether we can obtain the limit is not. $f$ being clearly permutation symmetric in the $\gamma_j$ we only have to consider the 8 cosets of $S_4$ in $W(D_4)$. This gives us 8 functions $f$. It turns out only 5 of these functions are linearly independent. However, in each point of space at least one of these five (in fact we only need 4 of them) is non-zero at this point, or at least at some points in an arbitrary small neighborhood of this point. In particular we can obtain 
a proper limit for $\Omega_{\lambda/\kappa}$ at all points in space.
\begin{lemma} 
Let $f$ be as defined in Proposition \ref{proplimomega}.
Consider the four functions 
\begin{align*}
f(\alpha,\beta;\gamma)&, & f_{12}(\alpha;\beta;\gamma) &= f(\alpha,\beta-\gamma_1-\gamma_2;-\gamma_1,-\gamma_2,\gamma_3,\gamma_4), \\ f_{13}(\alpha;\beta;\gamma) &= f(\alpha,\beta-\gamma_1-\gamma_3;-\gamma_1,\gamma_2,-\gamma_3,\gamma_4), &f_{1234}(\alpha;\beta;\gamma) &= f(\alpha,\beta-\sum_{r=1}^4 \gamma_r;-\gamma_1,-\gamma_2,-\gamma_3,-\gamma_4). 
\end{align*}
Then for any point $(\alpha,\beta;\gamma)$ at least one of these four functions is not locally constant and, hence, not locally zero.
\end{lemma}
\begin{proof}
As $f$ is piecewise linear we can differentiate almost everywhere; in particular we may assume that all four functions are differentiable at the point
$(\alpha,\beta;\gamma)$. Moreover, if all four functions are locally constant, the derivative (in any direction) has to vanish. 
In particular we get 
\begin{multline*}
0 = \frac{d}{d\alpha} (f-f_{1234}) - (\frac{d}{d\gamma_2}-\frac{d}{d\gamma_4})(f_{12}-f_{1234}) - (\frac{d}{d\gamma_1}-\frac{d}{d\gamma_3})(f-f_{13}) - 2(\frac{d}{d\gamma_1}-\frac{d}{d\gamma_2})(f-f_{12}) \\ =4(g'(2\alpha) - g'(\alpha-\gamma_1)-g'(\alpha+\gamma_1)).
\end{multline*}
Now we can plug in $g'(x) = 1-2\{x\}$ (which holds if $x\not \in \mathbb{Z}$) and divide by 8 to get
\[
0 = -\frac12 -\{2\alpha\} + \{\alpha-\gamma_1\} + \{\alpha + \gamma_1\}.
\]
As the sum of the fractional parts in this equation is an integer, this equation cannot be satisfied, so our assumption that all derivatives vanish is false. Therefore at least one of the four functions must not be locally constant.
\end{proof}

The upshot is that for any vector $(\alpha,\beta;\gamma)$ we can determine the size of $\Omega_{\lambda/\kappa}$. 
Combining the previous lemmas we thus get
\begin{proposition}
There exist piecewise linear functions $f_1(\alpha,\beta;\gamma)$ and $f_2(\alpha,\beta;\gamma)$ such that
\[
val(\Omega_{\lambda/\kappa}(ap^{\alpha},bp^{\beta};v_r p^{\gamma_r};p,q)) = f_1(\alpha,\beta;\gamma)|\lambda| + f_2(\alpha,\beta;\gamma)|\kappa|.
\]
\end{proposition}
In fact we even have explicit, but very ugly, expressions for $f_1$ and $f_2$; indeed, using whichever of the four functions in the previous lemma is not locally zero at a point, will provide us with the values of $f_1$ and $f_2$. In specific cases (i.e. when we are given $\alpha$, $\beta$ and $\gamma$) we can with relative ease (by checking all four cases) find an expression for the valuation and leading coefficients of $\Omega_{\lambda/\kappa}$. 

\subsection{Limits of Biorthogonal functions}
Now we consider what limits we can obtain from the biorthogonal functions. We consider here limits of
\[
\tilde R_{\lambda}^{(n)} (z_i p^{\zeta}; t_0p^{\alpha_0}: t_1p^{\alpha_1},t_2p^{\alpha_2},t_3p^{\alpha_3};u_0p^{\gamma_0}, u_1p^{\gamma_1};q,t;p),
\]
where the parameters $z_i$, $t_r$, $u_r$, $q$, and $t$ are independent of $p$. 

If we consider the expansion \eqref{eqbiorthousingomega} in terms of $\Omega_{\lambda/\kappa}$, we can determine the valuation of the summands, in particular the summands have valuation $h_1|\lambda| + h_2 |\nu|$ for some piecewise linear functions
$h_1$ and $h_2$ of $\zeta$, $\alpha$ and $\gamma$. We can now use a similar method as for $\Omega_{\lambda/\kappa}$ to determine the valuation of the limit. However because we know the limits of the interpolation functions often form a linearly independent set of functions of $z_i$ we have an extra way of seeing if cancellation might occur when the valuation of all summands is equal. This allows us to simplify the argument.
\begin{proposition}
Suppose $t^{2(n-1)}t_0t_1t_2t_3u_0u_1=pq$. 
Write $t_r=\tilde t_r p^{\alpha_r}$, $u_r=\tilde u_r p^{\gamma_r}$, $z_i=\tilde z_i p^{\zeta}$, and
$v=\tilde v p^{\nu}$, where $\tilde t_r$, $\tilde u_r$, $\tilde z_i$, $t$, $q$, and $\tilde v$ are independent of $p$. Let $h_1$ and $h_2$ be piecewise linear functions such that (with $s=1/\sqrt{pqu_0u_1}$)
\begin{multline*}
val\left(
\Omega_{\lambda/\nu}(s, \frac{u_0s}{t^{n-1}t_0 }; 
\frac{pqs}{t_0t_1t^{n-1}}, \frac{pqs}{t_0t_2t^{n-1}}, \frac{pqs}{t_0t_3t^{n-1}}, u_0vs;q,t;p)
C_\nu^0( \frac{pq}{u_0t_0}, \frac{pq}{u_0t_1}, \frac{pq}{u_0t_2}, \frac{pq}{u_0t_3}) 
R_{\nu}^{*(n)}(z_i;v,u_0) \right)
\\ = h_1(\zeta;\alpha;\gamma;\nu)|\lambda|+h_2(\zeta;\alpha;\gamma;\nu) |\nu|.
\end{multline*}
\begin{itemize}
\item If $h_2<0$ then 
\begin{multline*}
lc(\tilde R_{\lambda}^{(n)}(z_i;t_0{}:{}t_1,t_2,t_3;u_0,u_1)) \\= 
lc\left(\frac{C_{\lambda}^0(\frac{pqt^{n-1}t_0}{u_0},\frac{pq}{u_0t_0}, \frac{pq}{u_0t_1}, \frac{pq}{u_0t_2}, \frac{pq}{u_0t_3})}{C_{\lambda}^0(\frac{pq}{vu_0}, t^{n-1}t_0t_1,t^{n-1}t_0t_2,t^{n-1}t_0t_3, \frac{1}{t^{n-1}t_0u_1})} 
\frac{C_{\lambda}^+(\frac{1}{u_0u_1})}
{C_{\lambda}^+(    \frac{t^{n-1}  v}{u_0} )}
R_{\lambda}^{*(n)}(z_i;v,u_0) \right) 
\end{multline*}
and the valuations are also equal.
\item If $h_2=0$ and $(\zeta;\nu;\gamma_0)$ is not in (a shift of the) interior of the octahedron (from Proposition \ref{propliminterpolation}) then 
\begin{align*}
lc(\tilde R_{\lambda}^{(n)}(z_i;t_0{}:{}t_1,t_2,& t_3;u_0,u_1))  = 
\sum_{\nu \subset \lambda} lc\bigg(
\frac{C_{\lambda}^0(\frac{pqt^{n-1}t_0}{u_0})C_\nu^0( \frac{pq}{u_0t_0}, \frac{pq}{u_0t_1}, \frac{pq}{u_0t_2}, \frac{pq}{u_0t_3})}{C_{\lambda}^0(\frac{pq}{vu_0}, t^{n-1}t_0t_1,t^{n-1}t_0t_2,t^{n-1}t_0t_3, \frac{1}{t^{n-1}t_0u_1})} \\ & \qquad  \times 
\Omega_{\lambda/\nu}(s, \frac{u_0s}{t^{n-1}t_0 }; 
\frac{pqs}{t_0t_1t^{n-1}}, \frac{pqs}{t_0t_2t^{n-1}}, \frac{pqs}{t_0t_3t^{n-1}}, u_0vs;q,t;p)
R_{\nu}^{*(n)}(z_i;v,u_0)\bigg)
\end{align*}
and the valuation of the biorthogonal function equals the valuation of each of the summands.
\item If $h_2>0$ then 
\begin{multline*}
lc(\tilde R_{\lambda}^{(n)}(z_i;t_0{}:{}t_1,t_2,t_3;u_0,u_1)) = 
lc\bigg(\frac{C_{\lambda}^0(\frac{pqt^{n-1}t_0}{u_0})}{C_{\lambda}^0(\frac{pq}{vu_0}, t^{n-1}t_0t_1,t^{n-1}t_0t_2,t^{n-1}t_0t_3, \frac{1}{t^{n-1}t_0u_1})}
\\ \times \Omega_{\lambda/0}(s, \frac{u_0s}{t^{n-1}t_0 }; 
\frac{pqs}{t_0t_1t^{n-1}}, \frac{pqs}{t_0t_2t^{n-1}}, \frac{pqs}{t_0t_3t^{n-1}}, u_0vs;q,t;p) \bigg)
\end{multline*}
and the valuations are also equal.
\end{itemize}
\end{proposition}
It should be noted that we do not cover all possible cases in the proposition above, in particular we do not claim anything for the case $h_2= 0$ and $(\zeta;\nu;\gamma_0)$ in (a shift of the) interior of the octahedron (from Proposition \ref{propliminterpolation}). 
\begin{proof}
If $h_2<0$ then the valuation of the summands in the expansion \eqref{eqbiorthousingomega} of the biorthogonal function in terms of $\Omega_{\lambda/\nu}$ is minimized at $\nu=\lambda$, so we know the leading coefficient is just the leading coefficient of the $\nu=\lambda$ term. We simplified that term by using that $\Omega_{\lambda/\lambda}$ is a sum of just a single term, and again using the equation for $\binom{\lambda}{\lambda}$. If on the other hand $h_2>0$ then the leading coefficient of the biorthogonal function equals the leading coefficient of the $\nu=0$ summand. 

Finally if $h_2=0$ the valuations of all summands are equal, however if $(\zeta;\nu;\gamma_0)$ is not in (a shift of) the interior of the octahedron, we see that the sum of the leading coefficients of the summands cannot vanish, as the limits of the interpolation functions form a linearly independent family (as functions of $z$). Thus the valuation of our biorthogonal function must equal the valuation of the individual summands, and the leading coefficient equals the sum of the leading coefficients of the summands. 
\end{proof}
Now note that we can always choose $\nu$ such that $(\zeta;\nu;\gamma_0)$ is not in the interior of the octahedron, for example by taking $\nu=\zeta$. In those cases this proposition exactly tells us what the limit of the biorthogonal function is. In particular, in that case, we see that the limits form an independent family of functions of the $z_i$ if and only if $h_2\leq 0$, which is equivalent to the condition that the limit of $\tilde R_{\lambda}^{(n)}$ depends on $z_i$ at all (for any $\lambda \neq 0$). Thus we can determine which limits form such a family (and therefore have a shot at being one part in a pair of biorthogonal functions) by considering whether the limits of the univariate biorthogonal functions with those parameters depend on $z$. And this is precisely the situation we studied in \cite[Section 4]{vdBRuniv}.

When writing down explicit limits, it is better to first try whether the cases $v=t_r$ and $\nu = \alpha_r$, ($0\leq r\leq 3$) work, as this would lead to an expression of the limit as a single series (over partitions) (using the evaluation from Theorem \ref{thmomegaeval}). It might be that the vector $(\zeta; \alpha_r;\gamma_0)$ is always in the interior of the octahedron, in which case we have no choice but to use a different $v$ and expand the leading coefficient as a double series, but fortunately this happens only very rarely.

The final question is when the limit of these families of biorthogonal functions still form a biorthogonal family of functions. 
This can only be true if both families of biorthogonal families are still $z$-dependent, and if the valuation of the norms is correct. 
This condition is easily seen to be equivalent to this condition in the univariate case. Therefore the question when the limit can still form a biorthogonal system reduces to the same combinatorial problem as in the univariate case.
In particular \cite[Theorem 5.2]{vdBRuniv} also holds for multivariate biorthogonal functions. 
A complete list of all possible limits is given in loc. cit. Here we just like to remark that this includes multivariate versions of all orthogonal polynomials in the $q$-Askey scheme, and the correspondence of $\alpha$-vectors to families of polynomials in the $q$-Askey scheme is given by the table in \cite[Section 7]{vdBRuniv}. In the next section we consider another interesting case, the Pastro polynomials.

\section{Pastro Polynomials}\label{secpastro}
In this section we study the special points in the degeneration scheme for which the limits are families of biorthogonal polynomials (i.e. outside the $q$-Askey scheme where we have orthogonality). These polynomials are multivariate analogues of the polynomials studied by Pastro \cite{Pastro}. Specializing the parameters in the correct way gives us the only example of orthogonal polynomials on the unit circle (as opposed to the real line for the $q$-Askey scheme) contained in our degeneration scheme: the Macdonald polynomials \cite{Macdonald}.

The top level of these polynomials is $1111pp$, associated to the vector $\alpha=(-\frac14,0,\frac14,\frac12;0,\frac12)$ 
For the first function 
$\tilde R_{\lambda}^{(n)}(zp^{-\frac14};t_0 p^{-\frac14}, t_1,t_2p^{\frac14},t_3p^{\frac12};u_0,u_1p^{\frac12})$ the valuation turns out to be zero, and we get three different representations, corresponding to the expansion \eqref{eqbiorthousingomega} with respect to the parameters $v=t_0$, $v=t_2$, and $v=t_3$, which make the $\Omega$ evaluate by the bulk difference equation. Note that the expansion with $v=t_1$ does not give a nice limit, as in this case, the limiting interpolation function are inside the octahedron, and hence do not depend on $z$.
Thus we get the following expansions (you can find the definitions of the limiting binomial coefficients and interpolation functions in Appendix \ref{apA}, in particular $R_{T,\mu}^{*(n)}$ are just the Macdonald polynomials).
\begin{align*}
P_{\lambda}^{(n)} &:= lc\left(  \tilde R_{\lambda}^{(n)}(zp^{-\frac14};t_0 p^{-\frac14}, t_1,t_2p^{\frac14},t_3p^{\frac12};u_0,u_1p^{\frac12}) \right) \\
&= \sum_{\mu\subset \lambda} \binom{\lambda}{\mu}_{F2} \frac{\tilde C_\mu^-(t) \tilde C_{\mu}^0(\frac{q}{t_1u_0})}{\tilde C_{\mu}^0(t^n) \tilde C_{\mu}^0(t^{n-1}t_0t_2)} R_{F2,\mu}^{*(n)}(z_i^{-1};t_0) q^{n(\mu')} t^{-2n(\mu)} \left(-t^{n-1} t_0\right)^{|\mu|} \\
&= \left( \frac{q}{t_1u_0} \right)^{|\lambda|}
\sum_{\mu\subset \lambda} \binom{\lambda}{\mu}_{F1} \frac{\tilde C_{\mu}^-(t) \tilde C_{\mu}^0(\frac{q}{t_1u_0})}
{\tilde C_\mu^0(t^n) \tilde C_{\mu}^0(t^{n-1}t_0t_2)} R_{F1,\mu}^{*(n)}(z_i^{-1},t_2) q^{-n(\mu')} \left( -\frac{t^{n-1}t_0t_1u_0}{q}\right)^{|\mu|}, \\
&= \frac{C_{\lambda}^0(\frac{1}{t^{n-1}t_3u_1})}{C_{\lambda}^0(t^{n-1}t_0t_2)} 
\left( \frac{ q}{t_1u_0}\right)^{|\lambda|}
\sum_{\mu \subset \lambda} \binom{\lambda}{\mu}_{E2, [\frac{1}{t^{n-1}t_3u_1}]} 
\frac{\tilde C_{\mu}^-(t) \tilde C_{\mu}^0(\frac{q}{t_1u_0})}{\tilde C_{\mu}^0(t^n) \tilde C_{\mu}^0(t^{n-1}t_3u_1)}
R_{T,\mu}^{*(n)}(z_i^{-1}) t^{-n(\mu)} t_2^{-|\mu|}, \\
\end{align*}
Univariately these three expansions correspond to 
\[
P_l^{(1)} = \frac{(1/t_3u_1;q)_l}{(t_0t_2;q)_l} \left( \frac{q}{t_1u_0}\right)^l {}_2\phi_1 \left( \begin{array}{c} \frac{q}{t_1u_0}, q^{-l} \\ q^{1-l}t_3u_1\end{array} ;q, \frac{q}{t_2z} \right),
\]
(if one considers the last expansion), and expressions of this series as ${}_3\phi_2$, and as a ${}_3\phi_1$, which are related to each other by \cite[(III.6) and (III.7)]{GR}. 

Notice that the parameters of the function only appear in certain combinations. In particular if we write $A=\frac{q}{t^{n-1}t_1u_0}$, $B=\frac{q}{t^{n-1}t_3u_1}$ and $w_i=\frac{q^{1/2}}{t^{n-1}t_2t_3u_1z_i }$ we can define (using the homogeneity of the Macdonald polynomials $R_{T,\mu}^{*(n)}$)
\begin{align*}
p_{\lambda}^{(n)}(w;A,B) &= 
\frac{C_{\lambda}^0(\frac{B}{q})}{C_{\lambda}^0(\frac{t^{n-1}AB}{q})} 
\left( A t^{n-1}\right)^{|\lambda|}
\sum_{\mu \subset \lambda} \binom{\lambda}{\mu}_{E2, [\frac{B}{q}]} 
\frac{\tilde C_{\mu}^-(t) \tilde C_{\mu}^0(At^{n-1})}{\tilde C_{\mu}^0(t^n) \tilde C_{\mu}^0(\frac{q}{B})}
R_{T,\mu}^{*(n)}(\frac{w_i q^{1/2}}{B}) t^{-n(\mu)} 
\end{align*}
and note 
\[
P_{\lambda}^{(n)}(z_i;t_0:t_1,t_2,t_3;u_0,u_1) = p_{\lambda}^{(n)}(\frac{q^{1/2}}{t^{n-1}t_2t_3u_1z_i };\frac{q}{t^{n-1}t_1u_0},\frac{q}{t^{n-1}t_3u_1}).
\]
For the right hand family (the functions with $u_0$ and $u_1$ interchanged) we notice that by the symmetries from Lemma \ref{lemRtildeabel}, and the permutation symmetry in the $t_r$'s we have 
\begin{align*}
\tilde R_{\lambda}^{(n)}(zp^{-\frac14};t_0 p^{-\frac14}, t_1,t_2p^{\frac14},t_3p^{\frac12};u_1p^{\frac12},u_0)
&= \tilde R_{\lambda}^{(n)}(zp^{\frac14};t_0 p^{\frac14}, t_1p^{\frac12},t_2p^{-\frac14},t_3;u_1,u_0p^{\frac12})
\\& = \tilde R_{\lambda}^{(n)}(z^{-1}p^{-\frac14};t_0 p^{\frac14}, t_1p^{\frac12},t_2p^{-\frac14},t_3;u_1,u_0p^{\frac12})
\\ &= 
\frac{\tilde R_{\lambda}^{(n)}(z^{-1}p^{-\frac14};t_2 p^{-\frac14}, t_3,t_0p^{\frac14},t_1p^{\frac12};u_1,u_0p^{\frac12})}
{\Delta_{\lambda}^0(p^{-\frac12}\frac{1}{u_0u_1}~|~ p^{\frac34} t^{n-1} t_0t_1, p^{\frac14} t^{n-1}t_0t_3, p^{-\frac34} \frac{1}{t^{n-1}t_0u_0}, p^{\frac34} \frac{qt^{n-1}t_2}{u_1}) }
\end{align*}
As a corollary we see that the limit on the right hand side is up to a constant equal to the limit on the left hand side with different parameters. Indeed we have
\begin{align*}
Q_{\lambda}^{(n)} &:= lc\left( \tilde R_{\lambda}^{(n)}(zp^{-\frac14};t_0 p^{-\frac14}, t_1,t_2p^{\frac14},t_3p^{\frac12};u_1p^{\frac12},u_0) \right) 
=  \left( \frac{1}{t^{2(n-1)}u_0t_0t_1t_2}\right)^{|\lambda|}  P_{\lambda}^{(n)}(z_i^{-1};t_2,t_3,t_0,t_1;u_1,u_0),
\end{align*}
and the corresponding valuation is 0.
In particular, setting
\[
q_{\lambda}^{(n)}(w_i;A,B) := p_{\lambda}^{(n)}(\frac{1}{w_i};B,A),
\]
we have 
\[
Q_{\lambda}^{(n)}(z_i;t_0:t_1,t_2,t_3;u_0,u_1) = 
\left( \frac{q}{t_3u_1}\right)^{-|\lambda|}
q_{\lambda}^{(n)}(\frac{q^{1/2}}{t^{n-1}t_2t_3u_1z_i };\frac{q}{t^{n-1}t_1u_0},\frac{q}{t^{n-1}t_3u_1}).
\]
Note that this is the same parameter correspondence as between $P_{\lambda}$ and $p_{\lambda}$. As already shown in 
\cite{vdBRuniv}, the univariate instances $p_{l}^{(1)}$ and $q_{l}^{(1)}$ are equal to the Pastro polynomials \cite[(3.1)]{Pastro}.

We would like to consider a special case. Indeed we want to specialize $t_3u_1t^{n-1}\to 1$, or equivalently $B\to q$. The simplest expression to do this in is in the expansion of $P_\lambda^{(n)}$ in terms of the Macdonald polynomials $R_{T,\mu}^{*(n)}$, and to use $p_{\lambda}$. We cannot just substitute $B=q$ as for example the term $\tilde C_{\mu}^0(q/B)$ in the denominator would vanish. However, taking the limit $p\to 0$ in \cite[(4.3)]{RainsBCn} with $a\to ap^{\alpha}$ (and $0<\alpha<1$) shows that
\[
\left. \binom{\lambda}{\mu}_{E2,[b]} \frac{\tilde C_{\lambda}^0(b)}{\tilde C_{\mu}^0(1/b)} \right|_{b=1} = \delta_{\lambda,\mu} 
(-aq)^{|\mu|-|\lambda|},
\]
so plugging this in we see that the sum reduces to a single term and we get 
\[
p_{\lambda}^{(n)}(w_i;A,q) = \frac{\tilde C_{\lambda}^-(t)}{\tilde C_{\lambda}^0(t^n)} \left( Aq^{-1/2}t^{n-1}\right)^{|\lambda|} t^{-n(\lambda)}
R_{T,\lambda}^{*(n)}(w_i),
\]
thus up to a constant these are just the Macdonald polynomials. If we also want $Q_{\lambda}^{(n)}$ to also be the Macdonald polynomials, we should also specialize $u_0t_1t^{n-1}\to 1$, or equivalently $A\to q$.

The limit of the biorthogonality relation gives us
\begin{align*}
\langle P_{\lambda}^{(n)}(\cdot; t_r;u_r;t,q), Q_{\mu}^{(n)}(\cdot;t_r;u_r;t,q) \rangle = 
\delta_{\lambda,\mu} t^{-2n(\lambda)} (t_1u_0)^{-|\lambda|} 
\frac{\tilde C_\lambda^-(q,t)}{\tilde C_{\lambda}^0(t^n,t_0t_2t^{n-1})}.
\end{align*}
where 
\begin{multline*}
\langle f,g\rangle = \frac{(q;q)^n}{n! (t;q)^n} 
\prod_{j=1}^n \frac{(t^j, t^{n-j} t_0t_2;q)}{(qt^{j-n}/t_1u_0,qt^{j-n}/t_3u_1;q)} 
\int_{C^n} f(z)g(z) \prod_{1\leq j<k\leq n} \frac{(z_j/z_k,z_k/z_j;q)}{(tz_j/z_k,tz_k/z_j;q)} 
\\ \times \prod_{j=1}^n \frac{\theta(qt^{1-n}z_j/t_0t_1u_0;q)}{(t_0/z_j,t_2z_j;q)} \frac{dz_j}{2\pi i z_j},
\end{multline*}
or in terms of $p_{\lambda}$ and $q_{\lambda}$ this becomes
\[
\langle p_{\lambda}^{(n)}(\cdot;A,B), q_{\mu}^{(n)}(\cdot;A,B) \rangle = \delta_{\lambda,\mu} t^{-2n(\lambda)} \left( \frac{ABt^{2(n-1)}}{q}\right)^{|\lambda|} \frac{\tilde C_{\lambda}^-(q,t)}{\tilde C_{\lambda}^0(t^n,\frac{ABt^{n-1}}{q})},
\]
where
\begin{multline*}
\langle f,g\rangle = \frac{(q;q)^n}{n! (t;q)^n} 
\prod_{j=1}^n \frac{(t^j, \frac{ABt^{2n-1-j}}{q};q)}{(At^{j-1},Bt^{j-1};q)} 
\int_{C^n} f(w)g(w) \prod_{1\leq j<k\leq n} \frac{(w_j/w_k,w_k/w_j;q)}{(tw_j/w_k,tw_k/w_j;q)} 
\\ \times \prod_{j=1}^n \frac{\theta(q^{1/2}w_j ;q)}{(Aw_j/q^{1/2} ,B/w_jq^{1/2};q)} \frac{dw_j}{2\pi i w_j}.
\end{multline*}
The contours here are chosen with conditions similar to the condition in Theorem \ref{thmbiorthogonality}, that is, $C=C^{-1}$ is a deformation of the unit circle, which includes $tC$ and the poles at $z_j=t_0q^{\mathbb{Z}_{\geq 0}}$, respectively $w_j = Bq^{-\frac12+\mathbb{Z}_{\geq 0}}$, and excludes the poles at $z_j=t_2^{-1}q^{\mathbb{Z}_{\leq 0}}$, respectively $w_j = A^{-1} q^{\frac12+\mathbb{Z}_{\leq 0}}$. 
In the univariate case this inner product relation reduces indeed to the biorthogonality \cite[(3.2)]{Pastro} of the Pastro polynomials\footnote{Pastro has a necessary condition on the parameters which we lack, the difference is that he insists that the contour is the unit circle, while we look at more general contours (which are allowed to make detours to include points outside the unit circle and exclude points inside the unit circle)}. Observe that the specialization $A,B\to q$ turns this measure into the measure of the Macdonald polynomials, in particular the univariate part of the measure is in that case just the constant 1.

We can also take the limit in the evaluation duality \eqref{eqduality}. There are two distinct evaluation dualities we can get in the limit, we can evaluate at a sequence based around $t_0$, or evaluate at a sequence based around $t_2$ (evaluating at $t_1$ or $t_3$ would not lead to an equation for Pastro polynomials).
Taking the leading coefficient in 
\[
\tilde R_{\lambda}^{(n)}(t_0t^{n-i} q^{\kappa_i} p^{-\frac14};t_0p^{-\frac14}:t_1,t_2p^{\frac14},t_3p^{\frac12};u_0,u_1p^{\frac12}) 
= \tilde R_{\kappa}^{(n)}(\hat t_0t^{n-i} q^{\lambda_i} p^{-\frac14};\hat t_0p^{-\frac14}:\hat t_1,\hat t_2p^{\frac14},\hat t_3p^{\frac12};\hat u_0,\hat u_1p^{\frac12}) 
\]
leads to 
\[
P_{\lambda}^{(n)}(t_0 t^{n-i} q^{\kappa_i}; t_0,t_1,t_2,t_3;u_0,u_1) = 
P_{\kappa}^{(n)}(\hat t_0 t^{n-i} q^{\lambda_i};\hat t_0,\hat t_1,\hat t_2,\hat t_3;\hat u_0, \hat u_1),
\]
where the dual parameters are given by
\[
\hat t_0:= \sqrt{t_0t_1t_2t_3/q}, \qquad 
\hat t_r:= \frac{t_0t_r}{\hat t_0}, \qquad 
\hat u_r:= \frac{\hat t_0u_r}{ t_0}.
\]
This is equivalent to 
\[
p_{\lambda}^{(n)}(\frac{q^{\frac12}}{A t^{n-i} q^{\kappa_i}};A,B) = p_{\kappa}^{(n)}(\frac{q^{\frac12}}{A t^{n-i} q^{\lambda_i}};A,B).
\]
The equation based at $t_2$ leads to 
\[
\left( \frac{1}{At^{n-1}} \right)^{|\lambda|} p_{\lambda}^{(n)}(\frac{B}{q^{\frac12}} t^{n-i} q^{\kappa_i};A,B) =
\left( \frac{1}{At^{n-1}} \right)^{|\kappa|} p_{\kappa}^{(n)}(\frac{B}{q^{\frac12}} t^{n-i} q^{\lambda_i};A,B).
\]
This second evaluation duality, when specialized at $B=q$, reduces to the evaluation duality of the Macdonald polynomials.

In this case it turns out that the evaluation duality gives an equation between the same functions (evaluated at different points). In general we would get an evaluation duality where both sides of the question are different functions (i.e. the dual parameters are not associated with the same limiting family of biorthogonal functions).

Let us finally consider the difference operators. First we consider the equations $D_q^{(n)}\tilde R_{\lambda}^{(n)} = \tilde R_{\lambda}^{(n)}$ where the parameters of the two $\tilde R_{\lambda}^{(n)}$'s are related by a $q$-shift. Expanding $D_q^{(n)}$ as a sum of $2^n$ terms one quickly sees that the valuation of the individual terms only depend on $\sigma \in \{\pm 1\}^n$ only through $|\sigma| = \sum_i \sigma_i$. In particular there are three options: Either the term with $\sigma=(1,1,\ldots,1)$ dominates, or the term with $\sigma=(-1,-1,\ldots,-1)$ dominates or all terms have the same valuation. In the latter case we have to worry about cancellation, where the valuation of the individual terms might be lower than the valuation of the sum. It turns out this does not happen, which we can easily check as we know the valuation of $D_q^{(n)}\tilde R_{\lambda}^{(n)}$. For example, taking the difference equation for $D_q^{(n)}(u_0,t_0,t_1)$ and taking the direct limit in each of the $2^n$ terms leads to 
\begin{align*}
& \frac{1}{\prod_{i=1}^{n} t^{2(n-i)} t_0^2t_1u_0 (1-t^{n-i}u_0t_1)} 
\sum_{\sigma\in \{\pm 1\}^n} 
\prod_{i:\sigma_i=1} u_0t_0t_1z_i(1-\frac{t^{n-1}u_0t_0t_1}{z_i})
\prod_{i:\sigma_i=-1} (-t^{n-1}u_0t_0t_1z_i) (1-\frac{t_0}{z_i}) 
\\ & \qquad \times t^{\binom{(|\sigma|+n)/2}{2}} \prod_{i:\sigma_i=-1} \prod_{j:\sigma_j=1} \frac{1-tz_j/z_i}{1-z_j/z_i}
P_{\lambda}^{(n)}(q^{\sigma_i/2}z_i;q^{1/2}t_0:q^{1/2}t_1,q^{-1/2}t_2,q^{-1/2}t_3;q^{1/2}u_0,q^{-1/2}u_1) \\ & =
P_{\lambda}^{(n)}(z_i;t_0:t_1,t_2,t_3;u_0,u_1),
\end{align*}
where the $t^{\binom{(|\sigma|+n)/2}{2}}$ comes from the cross terms with $\sigma_i=\sigma_j=1$.
Recall that there are $6=\binom{4}{2}$ different difference operators like this, one for each pair of $t$-parameters.
The difference equations look simpler when expressed in the rescaled polynomials $p_{\lambda}$ so in the list below we will use those.  
\begin{itemize}
\item $D_q^{(n)}(u_0,t_0,t_1)$:
\begin{align*}
& \frac{t^{-2\binom{n}{2}}}{\prod_{i=1}^{n} (1-qt^{1-i}/A)} 
\sum_{\sigma\in \{\pm 1\}^n} 
\prod_{i:\sigma_i=1} \frac{q^{1/2}}{Aw_i}(1-q^{1/2}w_i)
\prod_{i:\sigma_i=-1} t^{n-1} (1-\frac{q^{1/2}}{Aw_i}) 
\\ & \qquad \times t^{\binom{(|\sigma|+n)/2}{2}} \prod_{i:\sigma_i=-1} \prod_{j:\sigma_j=1} \frac{1-tw_j/w_i}{1-w_j/w_i}
p_{\lambda}^{(n)}(w_iq^{(3-\sigma_i)/2};A/q,Bq) =
p_{\lambda}^{(n)}(w_i;A,B),
\end{align*}
\item $D_q^{(n)}(u_0,t_0,t_2)$:
\begin{align*}
& \frac{t^{-\binom{n}{2}}}{\prod_{i=1}^{n} (1-\frac{t^{n-i}AB}{q})}
\sum_{\sigma \in \{\pm 1\}^n}
\prod_{i:\sigma_i=1} (1-\frac{B}{q^{1/2}w_i}) 
\prod_{i:\sigma_i=-1} t^{n-1}\frac{B}{q^{1/2}w_i} (1-\frac{Aw_i}{q^{1/2}})
\\ & 
 \times t^{\binom{(|\sigma|+n)/2}{2}} \prod_{i:\sigma_i=-1} \prod_{j:\sigma_j=1} \frac{1-tw_j/w_i}{1-w_j/w_i}
 p_{\lambda}^{(n)}(w_iq^{(1-\sigma_i)/2};A,Bq) = p_{\lambda}^{(n)}(w_i;A,B)
\end{align*}
\item $D_q^{(n)}(u_0,t_0,t_3)$, in this case the term with $\sigma_i=1$ for all $i$ is dominant in the limit, and the resulting ``difference'' equation of one term is the trivial equation $p_{\lambda}^{(n)}(w_i;A,B)= p_\lambda^{(n)}(w_i;A,B)$.
\item $D_q^{(n)}(u_0,t_1,t_2)$: In this case we have to be careful about the different normalization if we interchange $t_0$ and one of the other $t$-variables. This accounts for the $q^{-|\lambda|}$ term on the right hand side.
\begin{align*}
\frac{1}{\prod_{i=1}^{n} (1-\frac{qt^{1-i}}{A})}& \sum_{\sigma\in \{\pm 1\}^n}
\prod_{i:\sigma_i=1} (-\frac{q}{At^{n-1}})(1-\frac{B}{w_iq^{1/2}}) 
\prod_{i:\sigma_i=-1} (1-\frac{q^{1/2}B}{Aw_i}) \\ & 
\times t^{\binom{(|\sigma|+n)/2}{2}} \prod_{i:\sigma_i=-1} \prod_{j:\sigma_j=1} \frac{1-tw_j/w_i}{1-w_j/w_i}
p_{\lambda}^{(n)}(q^{(1-\sigma_i)/2} w_i;A/q,Bq) = 
q^{-|\lambda|} p_{\lambda}^{(n)}(w_i;A,B)
\end{align*}
\item $D_q^{(n)}(u_0,t_1,t_3)$ 
\begin{align*}
\frac{1}{\prod_{i=1}^{n} (1-\frac{qt^{1-i}}{A})}& \sum_{\sigma\in \{\pm 1\}^n}
\prod_{i:\sigma_i=1} (-\frac{q}{At^{n-1}}) t^{\binom{(|\sigma|+n)/2}{2}} \\ & 
\times  \prod_{i:\sigma_i=-1} \prod_{j:\sigma_j=1} \frac{1-tw_j/w_i}{1-w_j/w_i}
p_{\lambda}^{(n)}(q^{(1-\sigma_i)/2} w_i;A/q,B) = 
q^{-|\lambda|} \frac{\tilde C_{\lambda}^0(\frac{t^{n-1}AB}{q})}{\tilde C_{\lambda}^0(\frac{t^{n-1}AB}{q^2})}p_{\lambda}^{(n)}(w_i;A,B)
\end{align*}
\item $D_q^{(n)}(u_0,t_2,t_3)$. Again we have one dominating term, and the equation becomes trivial. 
\end{itemize}
Curiously enough, if we consider the first and fourth of these equations, we have two difference operators with different shifts of the parameter (either by $(q^2,q)$, respectively $(q,1)$) acting on the same function $p_{\lambda}^{(n)}(\cdot;A/q,Bq)$ giving the same result up to a generalized eigenvalue. This is in sharp contrast to the situation for general biorthogonal functions in our scheme, in which generalized eigenvalue problems only arise from products of pairs of such operators.
Another remark is that the fifth difference equation preserves $B$, so we can specialize $B=q$ on both sides and obtain a difference equation for Macdonald polynomials. This operator is Macdonald's original difference operator \cite[(3.2)]{Macdonald}, of which the Macdonald functions are eigenfunctions.

We have two more difference operators, namely $D_-$ and $D_+$. These also have proper limits and we get as limit for $D_-$ that
\begin{align*}
& \sum_{\sigma \in \{\pm 1\}^n} \frac{1}{\prod_{i=1}^n w_i}
\frac{q^{n/2}}{A^{n}} (-1)^{(-n-|\sigma|)/2}
  t^{-5\binom{n}{2}-(n-1) (|\sigma|+n)/2+\binom{(|\sigma|+n)/2}{2}}
\\ &\qquad  \times \prod_{i:\sigma_i=-1} \prod_{j:\sigma_j=1} \frac{1-tw_j/w_i}{1-w_j/w_i}
p_{\lambda+1^n}^{(n)}(q^{(1-\sigma_i)/2} w_i;A/q,B) =
\prod_{i=1}^n \frac{(1-\frac{A}{q}t^{i-1})(1-t^{i-n}q^{-\lambda_i-1})}{(1-\frac{AB}{q^2}t^{n-i})}
p_{\lambda}^{(n)}(w_i;A,B)
\end{align*}
The $D_+$ operator moreover gives us
\begin{align*}
\prod_{i=1}^n & \frac{At^{n-1}}{ (1-\frac{AB}{q}t^{n-i})}
\sum_{\sigma \in \{\pm 1\}^n}
\prod_{i:\sigma_i=1} \frac{w_i}{q^{1/2}  t^{n-1}} (1-\frac{B}{w_i q^{1/2}})
\prod_{i:\sigma_i=-1} (1-\frac{Aw_i}{q^{1/2}}) 
\\ & \times t^{\binom{(|\sigma|+n)/2}{2}}
\prod_{i:\sigma_i=-1} \prod_{j:\sigma_j=1} \frac{1-tw_j/w_i}{1-w_j/w_i}
p_{\lambda}^{(n)}(q^{(1-\sigma_i)/2} w_i  ;A,Bq)
= p_{\lambda+1^n}^{(n)}(w_i;A,B).
\end{align*}

\appendix

\section{Limits of interpolation functions}\label{apA}
In what follows we are going to define $c_{\epsilon,\lambda,\kappa}$ for different values of $\epsilon$. These are the coefficients used in the branching rules for different limits of the interpolation functions $R_{\epsilon,\lambda}^{*}$. This convention allows us to define these limiting interpolation functions simultaneously.
\begin{definition}
For any $\epsilon$ for which $c_{\epsilon,\lambda,\kappa}$ is defined we define the corresponding interpolation function $R_{\epsilon,\lambda}^{*(n)}(z_i;a,b;q,t)$ recursively by
setting $R_{\epsilon,0}^{*(0)}=1$ and for $\lambda \neq 0$ setting $R_{\epsilon,\lambda}^{*(0)}=0$, and using the branching rule
\begin{equation}
R_{\epsilon,\lambda}^{*(n+1)}( \ldots, z_i, \ldots, v;a,b;q,t)  = \sum_{\kappa: \kappa \prec' \lambda} c_{\epsilon,\lambda,\kappa}
 R_{\epsilon,\kappa}^{*(n)}(\ldots,z_i,\ldots ;a,b;q,t).
\end{equation}
\end{definition}

As an important part in the definition of $c_{\epsilon,\lambda,\kappa}$ the following expression often arises.
\begin{definition}
For $\mu \not \prec' \lambda$ we set 
$\binon{\lambda}{\mu}_{[a,t];q,t} = 0$. 
For $\mu \prec' \lambda$ we define
\begin{align*}
\binon{\lambda}{\mu}_{[a];q,t} &= \prod_{\substack{(i,j)\in \lambda \\ \lambda_j'=\mu_j'}}
\frac{(1-q^{\lambda_i-j+1}t^{\lambda_j'-i})(1-q^{\mu_i-j}t^{1+\mu_j'-i} )}
{(1-q^{\mu_i-j+1}t^{\mu_j'-i})(1-q^{\lambda_i-j}t^{1+\lambda_j'-i} )}
\frac{\prod_{\substack{(i,j)\in \mu \\ \lambda_j'\neq \mu_j'}} (1-q^{\lambda_i+j}t^{1-\lambda_j'-i} a)(1-q^{\mu_i+j-1}t^{1-\mu_j'-i} a)}
{\prod_{\substack{(i,j)\in \lambda \\ \lambda_j' \neq \mu_j'}} (1-q^{\mu_i+j}t^{-\mu_j'-i} a)(1-q^{\lambda_i+j-1}t^{2-\lambda_j'-i} a)}
\end{align*}
\end{definition}

The definition of this new binomial coefficient was inspired by the following lemma.
\begin{lemma}\label{lemmalimbinon}
Let $\mu \prec' \lambda$. 
Then we find for $0<\alpha<1$
\begin{align*}
val \left(\binon{\lambda}{\mu}_{[a,t];q,t;p} \right) &= 0, \qquad \qquad 
val\left(\binon{\lambda}{\mu}_{[p^{\alpha} a,t];q,t;p} \right) =\alpha( |\lambda|-|\mu|), \\  
  lc\left( \binon{\lambda}{\mu}_{[a,t];q,t;p} \right)& =\frac{\tilde C_{\lambda}^-(t)\tilde C_{\lambda}^+(a)}{\tilde C_{\mu}^-(t)\tilde C_{\mu}^+(a/t)} 
 \frac{\left( -\frac{aq}{t}\right)^{|\lambda|}q^{n(\lambda')}t^{-2n(\lambda) }}{\left(-\frac{aq}{t}\right)^{|\mu|}q^{n(\mu')}t^{-2n(\mu)}} 
\binon{\lambda}{\mu}_{[a];q,t} \\
lc\left(\binon{\lambda}{\mu}_{[p^{\alpha} a,t];q,t;p} \right)& = 
\frac{\tilde C_{\lambda}^-(t)}{\tilde C_{\mu}^-(t)} 
 \frac{\left( -\frac{aq}{t}\right)^{|\lambda|}q^{n(\lambda')}t^{-2n(\lambda) }}{\left(-\frac{aq}{t}\right)^{|\mu|}q^{n(\mu')}t^{-2n(\mu)}} 
\binon{\lambda}{\mu}_{[0];q,t} 
\end{align*}
\end{lemma}
\begin{proof}
It is convenient to first rewrite the elliptic binomial coefficient as 
\begin{align*}
\binon{\lambda}{\mu}_{[a,t];q,t;p}
&= \frac{C_{\lambda}^-(t)C_{\lambda}^+(a)}{C_{\mu}^-(t)C_{\mu}^+(a/t)}
\prod_{\substack{(i,j)\in \lambda \\ \lambda_j'=\mu_j'}}
\frac{\theta(q^{\lambda_i-j}t^{\lambda_j'-i} pq,q^{\mu_i-j}t^{1+\mu_j'-i} ;p)}
{\theta(q^{\mu_i-j}t^{\mu_j'-i} pq,q^{\lambda_i-j}t^{1+\lambda_j'-i} ;p)}
\\& \qquad \times 
\frac{\prod_{\substack{(i,j)\in \mu \\ \lambda_j'\neq \mu_j'}}
\theta(q^{\lambda_i+j-1}t^{1-\lambda_j'-i} pq a,q^{\mu_i+j-1}t^{1-\mu_j'-i} a;p)}{\prod_{\substack{(i,j)\in \lambda \\ \lambda_j' \neq \mu_j'}}\theta(q^{\mu_i+j-1}t^{-\mu_j'-i} pqa,q^{\lambda_i+j-1}t^{2-\lambda_j'-i} a;p)}
\\ &= 
\frac{(-a)^{|\lambda|}}{(-a)^{|\mu|}} \frac{q^{n(\lambda')+|\lambda|}}{q^{n(\mu')+|\mu|}}
\frac{t^{2n(\mu)+|\mu|}}{t^{2n(\lambda) + |\lambda|}}
\frac{C_{\lambda}^-(t)C_{\lambda}^+(a)}{C_{\mu}^-(t)C_{\mu}^+(a/t)}
\prod_{\substack{(i,j)\in \lambda \\ \lambda_j'=\mu_j'}}
\frac{\theta(q^{\lambda_i-j+1}t^{\lambda_j'-i},q^{\mu_i-j}t^{1+\mu_j'-i} ;p)}
{\theta(q^{\mu_i-j+1}t^{\mu_j'-i},q^{\lambda_i-j}t^{1+\lambda_j'-i} ;p)}
\\& \qquad \times 
\frac{\prod_{\substack{(i,j)\in \mu \\ \lambda_j'\neq \mu_j'}}
\theta(q^{\lambda_i+j}t^{1-\lambda_j'-i} a,q^{\mu_i+j-1}t^{1-\mu_j'-i} a;p)}{\prod_{\substack{(i,j)\in \lambda \\ \lambda_j' \neq \mu_j'}}  \theta(q^{\mu_i+j}t^{-\mu_j'-i} a,q^{\lambda_i+j-1}t^{2-\lambda_j'-i} a;p)}
\end{align*}
The first identity is obtained by plugging in the definitions \eqref{eqdefcm}, respectively \eqref{eqdefcp}, of $C_{\lambda}^-$ and $C_{\lambda}^+$ as products over boxes in $\lambda$, and distributing those products over the terms $\prod_{(i,j) \in \lambda: \lambda_j'=\mu_j'}$ and 
$\prod_{(i,j) \in \lambda : \lambda_j'\neq\mu_j'}$. It should be observed that $\{(i,j) \in \lambda ~|~ \lambda_j'=\mu_j'\} = \{(i,j) \in \mu ~|~ \lambda_j'=\mu_j'\}$. The second equality then follows by using $\theta(px;p) = -\frac{1}{x} \theta(x;p)$ for all theta functions which visibly contain a $p$ in the argument. The resulting prefactors can be simplified to the term given by some combinatorial arguments. The two trickiest ones are that
\[
\sum_{(i,j) \in \mu} \lambda_i = \sum_i \sum_{1\leq j\leq \mu_i} \lambda_i = \sum_i \mu_i \lambda_i = \sum_{(i,j) \in\lambda} \mu_i
\]
and 
\[
\sum_{\substack{(i,j) \in \lambda \\ \lambda_j' \neq \mu_j'}} \mu_j' -
\sum_{\substack{(i,j) \in \mu \\ \lambda_j' \neq \mu_j'}} (\lambda_j' -1)
= \sum_{ (i,j) \in \lambda/\mu} (\lambda_j'-1) = 
\sum_{ (i,j) \in \lambda/\mu} (i-1) = 
\sum_{(i,j) \in \lambda} (i-1) - \sum_{(i,j) \in \mu} (i-1) = 
n(\lambda) - n(\mu).
\]
In the latter calculation we use that if $\lambda_j'\neq \mu_j'$ then $\lambda_j'=\mu_j'+1$ (as $\mu \prec' \lambda$), and that if 
$(i,j) \in \lambda/\mu$ then $\lambda_j'=i$ as $(i,j)$ must be the lowest box in its column (again as $\mu \prec' \lambda$).

We have now written the elliptic $\binon{\lambda}{\mu}$ as  a product where the leading coefficients of the terms are clear, so the two limits follow immediately.
%
\end{proof}

The case $\binon{\lambda}{\mu}_{[0]}$ is closely related to the coefficient in the branching rule for Macdonald polynomials, indeed
\begin{lemma}\label{lembinonpsi}
We have
\[
\binon{\lambda}{\mu}_{[0]}= \psi_{\lambda/\mu},
\]
where $\psi_{\lambda/\mu}$ is as given in \cite[page 341]{Macdonald} of Macdonald's book.
\end{lemma}
\begin{proof}
Notice that in the notation of \cite{Macdonald} we have $b_\mu(s)=b_\lambda(s)$ whenever 
$s\in \lambda- R_{\lambda/\mu} -C_{\lambda/\mu}$. Thus we can rewrite $\psi_{\lambda/\mu}$ in terms of our notation as
\[
\psi_{\lambda/\mu} = \prod_{s\in R_{\lambda/\mu}-C_{\lambda/\mu}} \frac{b_{\mu}(s)}{b_{\lambda}(s)} 
= \prod_{s\in \lambda-C_{\lambda/\mu}} \frac{b_{\mu}(s)}{b_{\lambda}(s)} 
= \prod_{\substack{(i,j)\in \lambda \\ \lambda_j' = \mu_j'}}
\frac{(1-q^{\lambda_i-j+1} t^{\lambda_j'-i})(1-q^{\mu_i-j}t^{\mu_j'-i+1})}
{(1-q^{\lambda_i-j}t^{\lambda_j'-i+1})(1-q^{\mu_i-j+1}t^{\mu_j'-i})}
= \binon{\lambda}{\mu}_{[0]}. \qedhere
\]
\end{proof}

This normalization is very useful in order to prove functions defined by the proper branching rules are monic.
\begin{proposition}
Suppose $c_{monic,\lambda,\kappa}= \binon{\lambda}{\kappa}_{[0]} P_n(v;\lambda,\kappa)$, where $P_n(v)$ is a monic polynomial in $v$ of degree $\lambda_n$, then 
the coefficient of $z^\lambda= z_1^{\lambda_1} \cdots z_n^{\lambda_n}$ in the resulting interpolation functions $R^*{(n)}_{monic,\lambda}$ 
equals 1.
\end{proposition}
\begin{proof}
The proof hinges on two observations. The first is that there is exactly one chain of partitions, which leads to terms $z^{\lambda}$, so the coefficient of $z^{\lambda}$ can be explicitly written as 
\[
[z^{\lambda}] R^*{(n)}_{monic,\lambda} = [z^{\lambda}] \prod_{i} c_{monic,(\lambda_1,\ldots, \lambda_{i-1}),(\lambda_1, \ldots, \lambda_i)}.
\]
The second is that if $\lambda = \mu\cdot k$, i.e., $\lambda_i = \mu_i$ for $i<n$ and $\lambda_n=k$ (hence $k\leq \mu_{n-1}$), then 
$\binon{\lambda}{\mu}_{[0]}=1$. Indeed in this case we see that if $(i,j) \in \lambda$ with $\lambda_j'=\mu_j'$, then $i<n$, so 
$\lambda_i=\mu_i$, Thus it follows that all terms in the defining product of $\binon{\lambda}{\mu}_{[0]}$ equal 1.
\end{proof}
In particular this proposition allows us have a well-defined normalization for all polynomial interpolation functions given below: we normalize them so that they are monic. 

Now we are ready to explicitly write down all limits. As mentioned in Section \ref{secliminterpolation} the limits under consideration are associated to the faces of different dimension of the tessellation with tiles as in Figure \ref{fig1}. Below we list all those faces (up to symmetry) and give the limiting interpolation functions.

\subsection{The vertex}
All the vertices give essentially the same limit by the shifting formulas. For simplicity we use the vertex $(0,0,0)$.

\begin{definition}
We define 
\[
c_{V,\lambda,\kappa} = \binon{\lambda}{\kappa}_{[t^n\frac{a}{b}];q,t} 
\frac{\tilde C_{\lambda}^0(t^n av^{\pm 1}) \tilde C_{\kappa}^0(\frac{q}{bt} v^{\pm 1})}
{\tilde C_{\kappa}^0(t^n av^{\pm 1} ) \tilde C_{\lambda}^0(\frac{q}{b}v^{\pm 1}) } t^{|\kappa|}.
\]
\end{definition}
The definition of $c_{V,\lambda,\kappa}$ automatically provides us with 
the interpolation functions $R_{V,\lambda}^{*(n)}$ and we get the following proposition.
\begin{proposition}
We obtain the limit%
\begin{align*}
val(R_{\lambda}^{*(n)}(z_i ;a,b;q,t;p)) & =0, \\
lc(R_{\lambda}^{*(n)}(z_i ;a,b;q,t;p)) &= R_{V,\lambda}^{*(n)}(z_i;a,b;q,t) \frac{\tilde C_{\lambda}^0(\frac{aq}{bt})\tilde C_{\lambda}^-(t) \tilde C_{\lambda}^+(t^{n-1}\frac{a}{b})}{\tilde C_{\lambda}^0(t^n)}
q^{2n(\lambda')} t^{-3n(\lambda)} \left(\frac{q^2 t^{n-1}}{b^2} \right)^{|\lambda|}
\end{align*}
\end{proposition}
\begin{proof}
The valuation $val(c_{\lambda,\kappa}(a,b;v))=0$ and the value of the leading coefficient $lc(c_{\lambda,\kappa}(a,b;v))$ 
 follow directly from the expressions for the binomial coefficient (Lemma \ref{lemmalimbinon}) and of the $\Delta_{\lambda}^0$'s (which follow from \eqref{eqlimc}).
The limit for the interpolation function then follows from the branching rule using induction to $n$. 

It could be that the newly defined $R_{V,\lambda}^{*(n)}$ is identically zero, in which case the valuation of the interpolation functions would be more than 0, the value given here, and the expression for the leading coefficient would be completely wrong. Fortunately we have Corollary \ref{cornonzero} below, which shows this is not the case.
\end{proof}
The proofs for the other limits of the interpolation functions are essentially the same as this one, so we will omit them from now on.
To get an idea of the kind of coefficients we have here, it should be observed that 
\[
\frac{\tilde C_{\lambda}^0(xv)}{\tilde C_{\kappa}^0(xv)} = 
\prod_{(i,j) \in \lambda/\kappa} (1-q^{j-1}t^{1-i}xv),
\] 
so it is a polynomial in $v$ of degree $|\lambda|-|\kappa|$, with leading coefficient 
\[
[v^{|\lambda|-|\kappa|}]\frac{\tilde C_{\lambda}^0(xv)}{\tilde C_{\kappa}^0(xv)}  = 
 \prod_{(i,j) \in \lambda/\kappa}
-xq^{j-1}t^{1-i}=(-x)^{|\lambda|-|\kappa|} q^{n(\lambda')-n(\kappa')} t^{n(\kappa)-n(\lambda)}.
\]
So in particular $\tilde C_{\lambda}^0(t^n av^{\pm 1})/\tilde C_{\kappa}^0(t^na v^{\pm 1})$ is a polynomial in 
$v+v^{-1}$ of degree $|\lambda|-|\kappa|$. For the terms involving $b$ we notice
\[
\frac{\tilde C_{\kappa}^0(\frac{q}{bt}v^{\pm 1})}{\tilde C_{\lambda}^0(\frac{q}{b} v^{\pm 1})} = 
\frac{1}{(\frac{q}{b} v^{\pm 1} ;q)_{\lambda_1}} \prod_{(i,j) \in \kappa/ (\lambda_2,\ldots, \lambda_n)} (1-q^{j-1}t^{1-i}\frac{qv}{bt})(1-q^{j-1}t^{1-i}\frac{q}{btv} ),
\]
thus we get a rational function with a polynomial in $v+v^{-1}$ of degree $\lambda_1+|\kappa|-|\lambda|$ in the numerator and the given polynomial 
$(\frac{q}{b} v^{\pm 1} ;q)_{\lambda_1}$ of degree $\lambda_1$ in the denominator.


\subsection{Edges}
For the edges we use the names $E1$, $E2$, $E3$ and $E4$, which correspond to 
\begin{itemize}
\item[E1] The edge connecting $(0,0,0)$ and $(\frac12,\frac12,\frac12)$;
\item[E2] The edge connecting $(0,0,0)$ and $(\frac12, -\frac12,\frac12)$;
\item[E3] The edge connecting $(0,0,0)$ and $(0,1,0)$;
\item[E4] The edge connecting $(0,0,0)$ and $(1,0,0)$.
\end{itemize}
The edges $E3$ and $E4$ are symmetric under the change $\zeta\to -\zeta$, while 
$E1$ and $E2$ change into the final two edges (mod lattice translations).

\begin{definition}
We define the coefficients
\begin{align*}
c_{E1,\lambda,\kappa} &= \binon{\lambda}{\kappa}_{[t^n \frac{a}{b}]} 
\frac{\tilde C_{\lambda}^0(t^n \frac{a}{v}) \tilde C_{\kappa}^0(\frac{qv}{bt})}{\tilde C_{\kappa}^0(t^n \frac{a}{v}) \tilde C_{\lambda}^0(\frac{qv}{b})}
\left(\frac{v}{b}\right)^{|\lambda|-|\kappa|} \\
c_{E2,\lambda,\kappa} &= \binon{\lambda}{\kappa}_{[0]} \frac{\tilde C_{\lambda}^0(t^n\frac{a}{v})\tilde C_{\kappa}^0(\frac{q}{btv})}
{\tilde C_{\kappa}^0(t^n \frac{a}{v}) \tilde C_{\lambda}^0(\frac{q}{bv})} t^{|\kappa|} \\
c_{E3,\lambda,\kappa} &= \binon{\lambda}{\kappa}_{[0]} \frac{\tilde C_{\lambda}^0(t^n a v^{\pm 1})}{\tilde C_{\kappa}^0(t^n av^{\pm 1})} 
\frac{q^{-n(\lambda')} t^{n(\lambda)} (-at^n)^{-|\lambda|}}{q^{-n(\kappa')} t^{n(\kappa)} (-at^n)^{-|\kappa|}} \\
c_{E4,\lambda,\kappa} &= \binon{\lambda}{\kappa}_{[0]} \frac{\tilde C_{\kappa}^0(\frac{q}{bt} v^{\pm 1})}{\tilde C_{\lambda}^0(\frac{q}{b} v^{\pm 1})} t^{|\kappa|}
\end{align*}
\end{definition}

\begin{proposition}
We have for $0<\alpha<1$ for the valuations
\begin{align*}
val(R_{\lambda}^{*(n)}(z_ip^{\alpha/2};ap^{\alpha/2},bp^{\alpha/2};q,t;p)&=0, & val(R_{\lambda}^{*(n)}(z_ip^{\alpha/2};ap^{\alpha/2},bp^{-\alpha/2};q,t;p)&=\alpha |\lambda|, \\ 
val(R_{\lambda}^{*(n)}(z_ip^{1/2};ap^{1/2},bp^{1/2-\alpha};q,t;p)&=\alpha |\lambda|, &
val(R_{\lambda}^{*(n)}(z_i;ap^{\alpha},b;q,t;p)&=0, 
\end{align*}
and for the leading coefficients
\begin{align*}
lc(R_{\lambda}^{*(n)}(z_ip^{\alpha/2};ap^{\alpha/2},bp^{\alpha/2};q,t;p)&=R_{E1,\lambda}^{*(n)}(z_i;a,b;q,t) \frac{\tilde C_{\lambda}^-(t) \tilde C_{\lambda}^+(t^{n-1} \frac{a}{b})\tilde C_{\lambda}^0(\frac{aq}{bt})}{\tilde C_{\lambda}^0( t^n)}
\left( -qt^{n-1}\right)^{|\lambda|} q^{n(\lambda')} t^{-2n(\lambda)},  \\
lc(R_{\lambda}^{*(n)}(z_ip^{\alpha/2};ap^{\alpha/2},bp^{-\alpha/2};q,t;p)&=R_{E2,\lambda}^{*(n)}(z_i;a,b;q,t) 
\frac{\tilde C_{\lambda}^-(t)}{\tilde C_{\lambda}^0(t^n)} 
\left( \frac{t^{n-1}q^2}{b^2}\right)^{|\lambda|} q^{2n(\lambda')} t^{-3n(\lambda)}, \\
lc(R_{\lambda}^{*(n)}(z_i;a,bp^{\alpha};q,t;p)) &= R_{E3,\lambda}^{*(n)}(z_i;a) \frac{\tilde C_{\lambda}^-(t)}{\tilde C_{\lambda}^0(t^n)} q^{n(\lambda')} t^{-2n(\lambda)} (-at^{n-1})^{|\lambda|}, \\
lc(R_{\lambda}^{*(n)}(z_i;ap^{\alpha},b;q,t;p)&=R_{E4,\lambda}^{*(n)}(z_i;b;q,t) 
\frac{\tilde C_{\lambda}^-(t)}{\tilde C_{\lambda}^0(t^n)} q^{2n(\lambda')} t^{-3n(\lambda)} \left( \frac{q^2 t^{n-1}}{b^2}\right)^{|\lambda|}.
\end{align*}

\end{proposition}
Note $R_{E4,\lambda}^{*(n)}$ does not depend on $a$, so we left it out of the notation. Similarly
$R_{E3,\lambda}^{*(n)}$ is independent of $b$. For $R_{E1, \lambda}^{*(n)}$ and $R_{E2,\lambda}^{*(n)}$ it is slightly more subtle. 
These satisfy an invariance of the following form: for arbitrary $x\in \mathbb{C}^*$ we have
\begin{equation}\label{eqE1inv}
R_{E1,\lambda}^{*(n)} (z_i x; ax,bx;q,t) = R_{E1,\lambda}^{*(n)} (z_i ; a,b;q,t), \qquad 
R_{E2,\lambda}^{*(n)} (z_i x; ax,b/x;q,t) = R_{E2,\lambda}^{*(n)} (z_i ; a,b;q,t).
\end{equation}
Thus, in essence, they too have one less variable than $R_{V,\lambda}^{*(n)}$. 

We would also like to remark that we normalized $R_{E3,\lambda}^{*(n)}$ to be a monic polynomial in $z_i+z_i^{-1}$. The interpolation function $R_{E3,\lambda}^{*(n)}$ is moreover equal to Okounkov's \cite{Okounkov} $BC_n$-symmetric interpolation function, via $R_{E3,\lambda}^{*(n)}(z_i;a)= (at^{n-1})^{|\lambda|}P_{\lambda}^*(z_i/at^i;q,t,a)$, as can be easily seen by comparing the branching rule given here with \cite[(5.2)]{Okounkov}.

\subsection{Triangles}
The different triangles are called $F1$, $F2$, $F3$, and $F4$ (for face).
\begin{itemize}
\item[F1] The triangle with vertices $(\frac12,-\frac12, \frac12)$, $(0,0,0)$ and $(\frac12,\frac12,\frac12)$;
\item[F2] The triangle with vertices $(0,0,0)$, $(0,1,0)$ and $(\frac12,\frac12,-\frac12)$;
\item[F3] The triangle with vertices $(\frac12,-\frac12, \frac12)$, $(0,0,0)$ and $(1,0,0)$ ;
\item[F4] The triangle with vertices $(0,0,0)$, $(1,0,0)$ and $(\frac12,\frac12,-\frac12)$;
\end{itemize}
Note we get all eight different triangles (up to lattice shifts) when we reflect these four by $\zeta\to -\zeta$.
\begin{definition}
We define the coefficients
\begin{align*}
c_{F1,\lambda,\kappa}&= \binon{\lambda}{\kappa}_{[0]} \frac{\tilde C_{\lambda}^0(t^n\frac{a}{v})}{\tilde C_{\kappa}^0(t^n \frac{a}{v})} v^{|\lambda|-|\kappa|}\\
c_{F2,\lambda,\kappa}&= \binon{\lambda}{\kappa}_{[0]} \frac{\tilde C_{\lambda}^0(t^n av)}{\tilde C_{\kappa}^0(t^n av)} \frac{(-t^na)^{-|\lambda|} q^{-n(\lambda')} t^{n(\lambda)}}
{(-t^n a)^{-|\kappa|} q^{-n(\kappa')} t^{n(\kappa)}} \\
c_{F3,\lambda,\kappa}&= \binon{\lambda}{\kappa}_{[0]} \frac{\tilde C_{\kappa}^0(\frac{q}{btv})}{\tilde C_{\lambda}^0(\frac{q}{bv})} t^{|\kappa|}, \\
c_{F4,\lambda,\kappa}&= \binon{\lambda}{\kappa}_{[0]} \frac{\tilde C_{\kappa}^0(\frac{q}{btv})}{\tilde C_{\lambda}^0(\frac{q}{bv})} v^{|\kappa|-|\lambda|}
\end{align*}
\end{definition}

\begin{proposition}
We have for $0<\alpha$ and $0<\beta$ and $\alpha+\beta<1$
\begin{align*}
val(R_{\lambda}^{*(n)}(z_ip^{\alpha/2+\beta/2};ap^{\alpha/2+\beta/2},bp^{\alpha/2-\beta/2};q,t;p)) &= \beta |\lambda| &
val(R_{\lambda}^{*(n)}(z_ip^{-\alpha/2};ap^{\alpha/2},bp^{\alpha/2+\beta};q,t;p)) &= 0 \\
val(R_{\lambda}^{*(n)}(z_ip^{\beta/2};ap^{\alpha+\beta/2},bp^{-\beta/2};q,t;p)) &= \beta |\lambda|, &val(R_{\lambda}^{*(n)}(z_ip^{-\beta/2};ap^{\alpha+\beta/2},bp^{\beta/2};q,t;p))
&=0
\end{align*}
\begin{align*}
lc(R_{\lambda}^{*(n)}(z_ip^{\alpha/2+\beta/2};ap^{\alpha/2+\beta/2},bp^{\alpha/2-\beta/2};q,t;p)) &= 
R_{F1,\lambda}^{*(n)}(z_i;a;q,t) \frac{\tilde C_{\lambda}^-(t)}{\tilde C_{\lambda}^0(t^n)} 
q^{n(\lambda')} t^{-2n(\lambda)} \left( - \frac{t^{n-1}q}{b} \right)^{|\lambda|}, \\
lc(R_{\lambda}^{*(n)}(z_ip^{-\alpha/2};ap^{\alpha/2},bp^{\alpha/2+\beta};q,t;p)) &= 
R_{F2,\lambda}^{*(n)}(z_i;a;q,t) \frac{\tilde C_{\lambda}^-(t)}{\tilde C_\lambda^0(t^n)} q^{n(\lambda')} t^{-2n(\lambda)} (-t^{n-1}a)^{|\lambda|}, \\
lc(R_{\lambda}^{*(n)}(z_ip^{\beta/2};ap^{\alpha+\beta/2},bp^{-\beta/2};q,t;p)) &=
R_{F3,\lambda}^{*(n)}(z_i;b;q,t) \frac{\tilde C_{\lambda}^-(t)}{\tilde C_{\lambda}^0(t^n)} q^{2n(\lambda')} t^{-3n(\lambda)}
\left( \frac{t^{n-1}q^2}{b^2}\right)^{|\lambda|}, \\
lc(R_{\lambda}^{*(n)}(z_ip^{-\beta/2};ap^{\alpha+\beta/2},bp^{\beta/2};q,t;p))
&= R_{F4,\lambda}^{*(n)}(z_i;b;q,t) \frac{\tilde C_{\lambda}^-(t)}{\tilde C_{\lambda}^0(t^n)} \left( -\frac{qt^{n-1}}{b}\right)^{|\lambda|}
q^{n(\lambda')} t^{-2n(\lambda)}
\end{align*}

\end{proposition}

Like for $R_{E1}$ and $R_{E2}$, these functions all satisfy a shifting equation, so they have one less variable than apparent. These equations are
\begin{align*}
R_{F1,\lambda}^{*(n)}(xz_i;ax) &= x^{|\lambda|}R_{F1,\lambda}^{*(n)}(z_i;a), & 
R_{F2,\lambda}^{*(n)}(xz_i;a/x) &= x^{|\lambda|} R_{F2,\lambda}^{*(n)}(z_i;a), \\
R_{F3,\lambda}^{*(n)}(xz_i;b/x) &= R_{F3,\lambda}^{*(n)}(z_i;b), & 
R_{F4,\lambda}^{*(n)}(xz_i;b/x) &= x^{-|\lambda|}R_{F4,\lambda}^{*(n)}(z_i;b).
\end{align*}

The interpolation function $R_{F1,\lambda}^{*(n)}$ correspond to the interpolation functions studied by Okounkov in \cite{Okounkov2}. The relation is
$R_{F1,\lambda}^{*(n)}(z_i;a) = P_{\lambda}^* (z_i/t^ia) (t^{n-1}a)^{-|\lambda|}$.

\subsection{Tetrahedra}
We denote the tetrahedron with vertices $(0,0,0)$, $(1,0,0)$, $(\frac12,\frac12,\frac12)$ and $(\frac12,-\frac12,\frac12)$
 by T. All tetrahedrons are shifts of this tetrahedron, or shifts of the reflection of this tetrahedron in the plane $\zeta=0$.

\begin{definition}
We define the coefficients
\[
c_{T,\lambda,\kappa} 
= \binon{\lambda}{\kappa}_{[0]} v^{|\lambda|-|\kappa|}
\]
\end{definition}

\begin{proposition}
We have for $0<\alpha, \beta,\gamma$ and $\alpha+\beta+\gamma<1$
\begin{align*}
val(R_{\lambda}^{*(n)}(z_ip^{\beta/2+\gamma/2};ap^{\alpha+\beta/2+\gamma/2},bp^{\beta/2-\gamma/2};q,t;p)) &= \gamma |\lambda| , \\
lc(R_{\lambda}^{*(n)}(z_ip^{\beta/2+\gamma/2};ap^{\alpha+\beta/2+\gamma/2},bp^{\beta/2-\gamma/2};q,t;p)) &= 
R_{T,\lambda,\kappa}(z_i;q,t) \frac{\tilde C_{\lambda}^-(t)}{\tilde C_{\lambda}^0(t^{n})}
q^{n(\lambda')} t^{-2n(\lambda)} \left(-\frac{t^{n-1}q}{b}\right)^{|\lambda|}
\end{align*}

\end{proposition}

These functions associated to the interiors of the tetrahedra are the famous Macdonald polynomials \cite{Macdonald}.

\begin{proposition}\label{propmcdpoly}
Let $P_{\lambda}$ denote the Macdonald polynomials. Then we have 
\[
P_{\lambda}(z_i;q,t) = R_{T,\lambda}^{*(n)}(z_i;q,t).
\]
\end{proposition}
\begin{proof}
The Macdonald polynomials can also be defined using a branching rule, indeed using \cite[(7.9)' and (7.14)']{Macdonald} we see that the coefficients in the branching rule for the Macdonald polynomials are exactly $\psi_{\lambda/\mu}$. By Lemma \ref{lembinonpsi} those are the coefficients in the branching rule for $R_{T,\lambda}$. The result now follows using induction.
\end{proof}
As a corollary we obtain the following important result.
\begin{corollary}\label{cornonzero}
The families of interpolation functions $\{R_{\epsilon,\lambda}^{*(n)}\}_{\lambda\in P_n}$ for $\epsilon$ equal to one of 
V, E1, E2, E3, E4, F1, F2, F3, F4, T form a basis for the symmetric functions in $z_i$. 
\end{corollary}
\begin{proof}
The family $\{R_{T,\lambda}^{*(n)}\}_{\lambda\in P_n}$ forms such a family as they are identical to Macdonald's interpolation functions;
the other families all degenerate to this family by taking the right limits, so they must also form such bases.
\end{proof}

\subsection{Octahedron}
The interior of the octahedron consists of two square pyramids and the square separating these pyramids. 
This gives us three limits, however all of them are independent of $z_i$. As functions in $z_i$ they are thus not interesting, however we can give a explicit expression for the limiting constant, which means we obtain some nontrivial combinatorial equations. 

Let us label the different polytopes by
\begin{itemize}
\item[S] The square with vertices $(0,0,0)$, $(\frac12,\frac12,\frac12)$, $(1,1,0)$ and $(\frac12,\frac12,-\frac12)$;
\item[P1] The pyramid with base S and apex $(0,1,0)$;
\item[P2] The pyramid with base S and apex $(1,0,0)$.
\end{itemize}

\begin{definition}
We define the coefficients
\begin{align*}
c_{S,\lambda,\kappa}&= \binon{\lambda}{\kappa}_{[t^n \frac{a}{b}]} t^{-|\kappa|}, &
c_{P1,\lambda,\kappa}&= \binon{\lambda}{\kappa}_{[0]} t^{|\kappa|}, & 
c_{P2,\lambda,\kappa}&= \binon{\lambda}{\kappa}_{[0]} t^{-|\kappa|}.
\end{align*}
\end{definition}

\begin{proposition}
We have the following evaluations
\begin{align*}
R_{S,\lambda}(a,b;q,t) &= \frac{\tilde C_{\lambda}^0(t^n)}{\tilde C_{\lambda}^-(t) \tilde C_{\lambda}^+(t^n \frac{a}{b}) \tilde C_{\lambda}^0(\frac{aq}{bt})}
t^{n(\lambda)} t^{-(n-1)|\lambda|}, \\
R_{P1,\lambda}(q,t) & = 
\frac{\tilde C_{\lambda}^0(t^n)}{\tilde C_{\lambda}^-(t)} t^{n(\lambda)}, \qquad \qquad 
R_{P2,\lambda}(q,t) =
\frac{\tilde C_{\lambda}^0(t^n)}{\tilde C_{\lambda}^-(t)} t^{n(\lambda)} t^{-(n-1)|\lambda|}
\end{align*}
Moreover we have for $\alpha,\beta,\gamma$ with $|\alpha| +|\beta|+|\gamma|<1$ the limit
\begin{align*}
val(R_{\lambda}^{*(n)}(z_i p^{\beta/2}; a p^{1/2 + \alpha/2+\gamma/2}, b p^{1/2+\alpha/2-\gamma/2};q,t;p))&=0, \\
lc(R_{\lambda}^{*(n)}(z_i p^{\beta/2}; a p^{1/2 + \alpha/2+\gamma/2}, b p^{1/2+\alpha/2-\gamma/2};q,t;p))&=1.
\end{align*}

\end{proposition}
\begin{proof}
First we obtain an expression for the limits of the interpolation functions using the branching rule as before. The case $\gamma=0$ corresponds to the square, while $\gamma<0$ corresponds to $P1$ and $\gamma>0$ to $P2$. This gives us the valuations of the interpolation functions in these cases, and expressions for the leading coefficient in, respectively, $R_{S,\lambda}$, $R_{P1,\lambda}$ and $R_{P2,\lambda}$. In particular we see that the leading coefficient is $z$-independent. But we also know the evaluation (Proposition \ref{propinterpoleval})  for general interpolation functions 
\[
R_{\lambda}^{*(n)}(xt^{n-i};a,b) = \Delta_{\lambda}^0(\frac{t^{n-1}a}{b}~|~ t^{n-1}ax, \frac{a}{x}).
\]
Plugging in $a\to ap^{1/2 + \alpha/2+\gamma/2}$, $b\to b p^{1/2+\alpha/2-\gamma/2}$, and $x\to x p^{\beta/2}$ in here, shows that the 
valuation on the right hand side is 0 and the leading coefficient 1. Thus we get another expression for the leading coefficient for special values of $z_i$. However, as the leading coefficient is (by inspection) independent of the $z_i$, it must be equal to this constant. The resulting equality gives us the evaluations of $R_{S,\lambda}$, $R_{P1,\lambda}$ and $R_{P2,\lambda}$.
\end{proof}


The value $R_{P1,\lambda}$ is equal to a special value of the Macdonald polynomials. Indeed, given a chain $0=\lambda^{(0)} \prec' \lambda^{(1)} \prec' \cdots \prec' \lambda^{(n)}=\lambda$, the coefficient of $\prod_i z_i^{|\lambda^{(i)}/\lambda^{(i-1)}|}$ in the Macdonald polynomial $P_{\lambda}$  equals 
$\prod_i \binon{\lambda^{(i)}}{\lambda^{(i-1)}}_{[0]}$. In $R_{P1}$ this product of binomial coefficients appears with a factor 
\[
\prod_i t^{|\lambda^{(i-1)}|} = \prod_i t^{\sum_{k=1}^{i-1} |\lambda^{(k)}/\lambda^{(k-1)}|} = \prod_k t^{(n-k)|\lambda^{(k)}/\lambda^{(k-1)}|}.
\]
Thus we obtain that $R_{P1,\lambda}=R_{T,\lambda}(t^{n-1}, t^{n-2},\ldots,1)$, and its evaluation thus corresponds to the famous principal evaluation of Macdonald polynomials. Likewise the evaluation of $R_{P2,\lambda}$ also corresponds to this same evaluation of Macdonald polynomials.

\subsection{Binomial coefficients}
Recall that the binomial coefficients were defined in terms of the interpolation functions. Obtaining the limits of 
the binomial coefficients is simply a question of plugging in the corresponding limits of the interpolation functions. These limits were all already discussed in \cite[Section 8]{RainsBCn}.

Proposition \ref{propbinomp} (which says the binomial coefficients are $p$-elliptic in $a$ and $b$) tells us that we only need consider limits for $\binom{\lambda}{\mu}_{[ap^{\alpha},bp^{\beta}];q,t;p}$ 
for $\alpha,\beta$ modulo the lattice $\mathbb{Z}^2$.
Notice that the binomial coefficients only use the interpolation functions with $\alpha=\zeta$, that is, in the plane at the front of Figure \ref{fig1}. The relevant picture of the different limits is given in Figure \ref{fig2}. The second labeling is included as it corresponds to the one in Figure \ref{fig1}; notice that we only have the front of the parallelepiped and reflected it left-right.

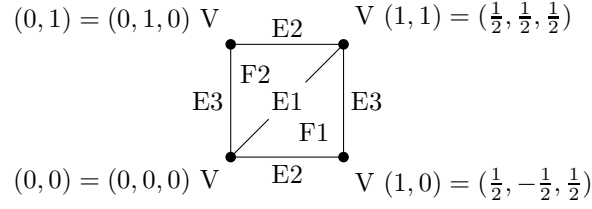
\begin{figure}

\begin{tikzpicture}[scale=0.3,vol/.style={circle,fill=black,minimum size=4pt,inner sep=0pt,
            outer sep=-1pt},label distance=0mm]

\node[label=below left:{$(0,0)=(0,0,0)$ V}] (lo) at (0,0) [vol] {};
\node[label=below right:{V $(1,0)=(\frac12,-\frac12,\frac12)$}] (ro) at (5,0) [vol] {};
\node[label=above left:{$(0,1)=(0,1,0)$ V}] (lb) at (0,5) [vol] {};
\node[label=above right:{V $(1,1)=(\frac12,\frac12,\frac12)$}] (rb) at (5,5) [vol] {};

\begin{scope}[label distance =-1.5mm]
\node[label=below:{E2}] (o) at (2.5,0) {}; 
\node[label=above:{E2}] (b) at (2.5,5) {}; 
\node[label=left:{E3}] (l) at (0,2.5) {}; 
\node[label=right:{E3}] (r) at (5,2.5) {}; 
\end{scope}
\node (m) at (2.5,2.5) {E1}; 
\node (f1) at (3.7,1.1) {F1};
\node (f2) at (1.1,3.7) {F2};

\draw (lo) -- (ro) -- (rb) -- (lb) -- (lo) -- (m) -- (rb) ;

\end{tikzpicture}

\caption{The different limits for the binomial coefficients. Points are labeled by $(\alpha,\beta) =(\alpha/2,\beta-\alpha/2,\alpha/2)$}\label{fig2}
\end{figure}

\begin{definition}
We define
\begin{align*}
\binom{\lambda}{\mu}_{V,[a,b];q,t} & := 
q^{3n(\mu')} t^{-4n(\mu)} \left( -\frac{q^3 t^{n-1}a^2}{b^2}\right)^{|\mu|}
\tilde \Delta_{\mu}^{(n)}(\frac{a}{b}) 
 \frac{\tilde C_{\mu}^0(\frac{1}{b},\frac{t^{-n}a q}{b}) \tilde C_{\mu}^-(t) \tilde C_{\mu}^+(\frac{a}{b})}
{\tilde C_{\mu}^0(aq,t^n)} \\ & \qquad \times 
R_{V,\mu}^{*(n)}(\sqrt{a} q^{\lambda_i}t^{1-i};t^{1-n}\sqrt{a},\frac{b}{\sqrt{a}})
\\
\binom{\lambda}{\mu}_{E1,[x];q,t} &:=  \left(-q^2 t^{n-1}x\right)^{|\mu|} q^{3n(\mu')} t^{-4n(\mu)}
\tilde \Delta_{\mu}^{(n)}(x)
\frac{\tilde C_{\mu}^-(t) \tilde C_{\mu}^+(x) \tilde C_{\mu}^0(t^{-n}qx)}{\tilde C_{\mu}^0(t^n)}
 R_{E1,\mu}^{*(n)} (q^{\lambda_i} t^{1-i};t^{1-n}, \frac{1}{x}) 
 \\
\binom{\lambda}{\mu}_{E2,[b];q,t} &:= q^{|\mu|} 
t^{n(\mu)} \frac{\tilde C_{\mu}^0(\frac{1}{b}) }{\tilde C_{\mu}^-(q)}  R_{E2,\mu}^{*(n)}(q^{\lambda_i} t^{1-i};t^{1-n}, b) 
 \\
\binom{\lambda}{\mu}_{E3,[a];q,t} &:= q^{n(\mu')} (-q\sqrt{a})^{|\mu|} \frac{1}{\tilde C_{\mu}^-(q) \tilde C_{\mu}^0(aq)}
R_{E3,\lambda}^{*(n)}(\sqrt{a} q^{\lambda_i} t^{1-i};t^{1-n}\sqrt{a})
\\
\binom{\lambda}{\mu}_{F1;q,t} &:=  
\frac{ t^{n(\mu)}}{\tilde C_{\mu}^-(q)}
R_{F1,\lambda}^{*(n)}( q^{\lambda_i} t^{1-i};t^{1-n})
\\
\binom{\lambda}{\mu}_{F2;q,t} &:= 
\frac{q^{n(\mu')} (-q)^{|\mu|}}{\tilde C_{\mu}^-(q)} 
R_{F2,\mu}^{*(n)}(\frac{1}{ q^{\lambda_i} t^{1-i}};t^{1-n}) 
\end{align*}
for $n> l(\lambda), l(\mu)$.
\end{definition}

\begin{proposition}
We have the following limits for $0<\alpha<\beta<1$
\begin{align*}
\lim_{p\to 0} \binom{\lambda}{\mu}_{[a,b];q,t;p} &= \binom{\lambda}{\mu}_{V,[a,b];q,t} &
\lim_{p\to 0} \binom{\lambda}{\mu}_{[ap^{\alpha},bp^{\alpha}];q,t;p} &= \binom{\lambda}{\mu}_{E1,[a/b];q,t} \\
\lim_{p\to 0} \binom{\lambda}{\mu}_{[ap^{\alpha},b];q,t;p} &= \binom{\lambda}{\mu}_{E2,[b];q,t} &
\lim_{p\to 0} \binom{\lambda}{\mu}_{[a,bp^{\alpha}];q,t;p} &= \binom{\lambda}{\mu}_{E3,[a];q,t} \\
\lim_{p\to 0} \binom{\lambda}{\mu}_{[ap^{\beta},bp^{\alpha}];q,t;p} &= \binom{\lambda}{\mu}_{F1;q,t} &
\lim_{p\to 0} \binom{\lambda}{\mu}_{[ap^{\alpha},bp^{\beta}];q,t;p} &= \binom{\lambda}{\mu}_{F2;q,t} 
\end{align*}
\end{proposition}
\begin{proof}
We plug in the limits of this section in Definition \ref{defbinom}. We simplify the results a bit (by using invariances of the form 
\eqref{eqE1inv}) to indicate that, for example, $\binom{\lambda}{\mu}_{E1,[a/b];q,t}$ only depends on $a/b$ (as opposed to $a$ and $b$).
\end{proof}

\end{document}